\documentclass[twocolumn]{autart} 
\usepackage{amsmath,amssymb,color}
\usepackage{graphicx,epsfig,framed}
\usepackage{mathptmx,times} 
\usepackage{epstopdf,textcase} 
\usepackage{soul,multirow,pifont} 
\usepackage{color,alltt}          
\usepackage[numbers,square,sort&compress]{natbib}
\usepackage{enumerate,siunitx}    

\definecolor{forestgreen}{rgb}{0.13, 0.55, 0.13}

\newtheorem{theorem}{Theorem}[section]
\newtheorem{corollary}[theorem]{Corollary}
\newtheorem{lemma}[theorem]{Lemma}

\newtheorem{proposition}[theorem]{Proposition}
\newtheorem{definition}[theorem]{Definition}
\newtheorem{example}[theorem]{Example}

\newtheorem{remark}[theorem]{Remark}
\numberwithin{equation}{section}

\newcommand{\rfb}[1]{\mbox{\rm
   (\ref{#1})}\ifx\undefined\stillediting\else:\fbox{$#1$}\fi}

\newfont{\roma}{cmr10 scaled 1200}

\newcommand{\fline}  {{\mathbb F}}

\newcommand{\rline}  {{\mathbb R}}
\newcommand{\tline}  {{\mathbb T}}

\newcommand{\PPP} {{\mathbf P}}
\newcommand{\SSS} {{\mathbf S}}

\newcommand{\dd}  {{\rm d}\hbox{\hskip 0.5pt}}
\renewcommand{\leq} {\leqslant}
\renewcommand{\geq} {\geqslant}
\newcommand{\Ascr} {{\mathcal A}}

\newcommand{\Dscr} {{\mathcal D}}

\newcommand{\Hscr} {{\mathcal H}}
\newcommand{\Lscr} {{\mathcal L}}
\newcommand{\Mscr} {{\mathcal M}}
\newcommand{\Nscr} {{\mathcal N}}
\newcommand{\Sscr} {{\mathcal S}}
\newcommand{\Uscr} {{\mathcal U}}
\newcommand{\Yscr} {{\mathcal Y}}
\newcommand{\mm}    {{\hbox{\hskip 0.5pt}}}
\newcommand{\m}     {{\hbox{\hskip 1pt}}}
\newcommand{\n}     {{\hbox{\hskip -5pt}}}
\newcommand{\nm}    {{\hbox{\hskip -3pt}}}

\newcommand{\bluff} {{\hbox{\raise 15pt \hbox{\mm}}}}
\newcommand{\bigbluff} {{\hbox{\raise 21pt \hbox{\mm}}}}
\newcommand{\bbigbluff} {{\hbox{\raise 35pt \hbox{\mm}}}}
\newcommand{\sbluff}{{\hbox{\raise 10pt \hbox{\mm}}}}

\renewcommand{\l}    {{\lambda}}

\newcommand{\FORALL} {{\hbox{$\hskip 11mm \forall \;$}}}
\newcommand{\rarrow} {\mathop{\rightarrow}}

\renewcommand{\half} {{\frac{1}{2}}}

\newcommand{\LEloc}[1] {{L^2_{\rm loc}([0,\infty);#1)}}
\newcommand{\NL}{{\rm NL}}
\newcommand{\AB}     {{A\& B}}
\newcommand{\CD}     {{C\& D}}
\newcommand{\imp}    {{\rm imp}}

\renewcommand{\theenumi}{\roman{enumi}}

\makeatletter
\renewcommand{\p@enumii}{}
\makeatother
\newcommand{\RomanNumeralCaps}[1]
    {\MakeUppercase{\romannumeral #1}}
    
\newcommand{\ULax}   {\boldsymbol{{\mathfrak{T}}}}
\newcommand{\GothA}  {\boldsymbol{{\mathfrak{A}}}}
\newcommand{\GothE}  {\boldsymbol{{\mathfrak{E}}}}

\font\bosy=cmbsy10
\def\conc{\hbox{\bosy \char '175}}
\def\ICON{\mathop{\conc}}

\newcommand{\Dom}[1]{\Dscr(#1)}

\newcommand{\loc}{{\rm loc}}
\renewcommand{\Re}{{\rm Re\,}}

\newcommand{\ipd}[2]{\langle #1, #2 \rangle}
\newcommand{\bbm}[1]{\left[\begin{matrix} #1 \end{matrix}\right]}
\newcommand{\sbm}[1]{\left[\begin{smallmatrix} #1
\end{smallmatrix}\right]}

\begin{document}
\begin{frontmatter}

\title{A class of incrementally scattering-passive nonlinear 
       systems\thanksref{footnoteinfo}\vspace{-3mm}}

\thanks[footnoteinfo]{The authors are working in the ITN network
ConFlex. This project is funded by the European Union's Horizon 2020
research and innovation programme under the Marie Sklodowska-Curie 
grant agreement no. 765579.\\
{\rm e-mail: shantanu@tauex.tau.ac.il (SS), gweiss@tauex.tau.ac.il (GW), marius.tucsnak@u-bordeaux.fr (MT).}}

\author[S. Singh]{Shantanu Singh}, 
\author[S. Singh]{George Weiss},            
\author[M. Tucsnak]{Marius Tucsnak}

\address[S. Singh]{School of Electrical Eng., Tel Aviv University, 
                   Ramat Aviv 69978, Israel.}
\address[M. Tucsnak]{Universit\'e de Bordeaux, Bordeaux INP, CNRS, 
                   F-33400 Talence, France.}
         

\begin{abstract} We investigate a special class of nonlinear infinite
dimensional systems. These are obtained by subtracting a nonlinear
maximal monotone (possibly multi-valued) operator $\Mscr$ from the
semigroup generator of a scattering passive linear system. While the
linear system may have unbounded linear damping (for instance,
boundary damping) which is only densely defined, the
nonlinear damping operator $\Mscr$ is assumed to be defined on the whole state space. We show that this new
class of nonlinear infinite dimensional systems is well-posed and
incrementally scattering passive. Our approach uses the theory of
maximal monotone operators and the Crandall-Pazy theorem about
nonlinear contraction semigroups, which we apply to a Lax-Phillips
type nonlinear semigroup that represents the whole system.
\end{abstract} 
\end{frontmatter}

\section{Introduction} \label{sec1} 
This paper deals with a special class of well-posed nonlinear systems whose state space is a Hilbert space. More precisely, the systems that we study are obtained from well-posed linear systems by adding a nonlinear damping term. We investigate the well-posedness of the resulting nonlinear system. To formulate our problem more clearly, we recall that a well-posed linear system $\Sigma$ with state space $X$, input space $U$ and output space $Y$ ($X,U$ and $Y$ are real Hilbert spaces) can be described as follows: $\Sigma$ has a semigroup generator $A:
\Dscr(A)\rarrow X$, a control operator $B:U\rarrow X_{-1}$ ($X_{-1}$ is
an extrapolation space containing $X$ densely) and an observation
operator $C:\Dscr(A)\rarrow Y$ (the precise meaning of all these
concepts will be recalled briefly in Section \ref{sec2}). There
exists an extension of $C$, denoted by $\overline{C}$ (not necessarily
unique), and an operator $D:U\rarrow Y$ (that depends on $\overline{C}
$) such that for sufficiently smooth input functions $u$ and a dense
subset of initial states $x_0$, the state trajectory of $\Sigma$,
denoted by $x$, and the output function of $\Sigma$, denoted by $y$,
satisfy for all $t\geq 0$ \vspace{-3mm}
\begin{equation} \label{Ax+Bu}
   \left\{\begin{array}{ll}\dot x(t) \m=\m Ax(t) + Bu(t) \m, \\
   y(t) \m=\m \overline{C}x(t) + Du(t) \m,  
   \end{array}\right. \vspace{-2mm}
\end{equation} 
and $x(0)=x_0$. Here, $A$ is an extension of the original semigroup 
generator, it maps from $X$ to $X_{-1}$.\vspace{-1mm} 

The system $\Sigma$ is called {\em scattering passive} if the 
following {\em energy balance inequality} holds for all $\tau>0$: 
\vspace{-1mm}
\begin{equation} \label{energy_balance}
   \|x(\tau)\|^2 + \int_0^\tau \|y(t)\|^2 \dd t \m\leq\m
   \|x(0)   \|^2 + \int_0^\tau \|u(t)\|^2 \dd t \m,\vspace{-2mm}
\end{equation}
for all the solutions of \rfb{Ax+Bu}. 
A short summary of the main facts about the well-posed linear systems,
in particular scattering passive systems will be given in Section
\ref{sec2}. \vspace{-1mm}

Let $X$ be a real Hilbert space. A set-valued operator $\Mscr$ defined
on $\Dscr(\Mscr)\subset X$ whose values are nonempty subsets of $X$ is
called {\em monotone} if $x_1,x_2\in\Dscr(\Mscr)$ and
$z_1\in\Mscr(x_1)$, $z_2\in\Mscr(x_2)$ implies that \vspace{-2mm}
$$\langle z_1-z_2, x_1-x_2\rangle_X \m\geq\m 0 \m.\vspace{-2mm}$$ 
The above operator is called {\em maximal monotone} if it has no
proper monotone extension (in the sense of inclusion of the graphs).
Every monotone operator has maximal monotone extensions. If $\Mscr$ is
maximal monotone, then for every $x\in\Dscr(\Mscr)$, the set
$\Mscr(x)$ is closed and convex and moreover the closure of $\Dscr
(\Mscr)$ is convex. If $\Mscr$ is linear and $\Dscr(\Mscr)=X$, then
$\Mscr$ being monotone is equivalent to $\Mscr$ being bounded and
$\Mscr+ \Mscr^*\geq 0$. For background on (maximal) monotone operators
we refer to \citep{Brezis,Brezis74,Browder68,Rocka,Show}. 
 A set-valued operator $Q$ is called {\em (maximal)
dissipative} if $-Q$ is (maximal) monotone. \vspace{-1mm}

This paper is about nonlinear infinite dimensional systems 
$\Sigma^\Mscr$ described by equations of the form\vspace{-2mm}
\begin{equation} \label{Boris}
   \dot x(t) \m\in\m Ax(t) - \Mscr(x(t)) + Bu(t) \m,\vspace{-2mm}
\end{equation}
\begin{equation} \label{Mirvis}
   y(t) \m=\m \overline{C}x(t) + Du(t) \m,
\end{equation}
where $A,B,\overline{C},D$ determine a scattering passive linear
system as in \rfb{Ax+Bu}, \rfb{energy_balance} and $\Mscr$ is a
maximal monotone (possibly set-valued) operator with
$\Dscr(\Mscr)=X$. A motivating example of such an infinite dimensional
system (representing a vibrating tower with a tuned mass damper) is in
Section \ref{sec3}. Another noteworthy example is presented in our
very recent conference paper \citep{ShWeTu:21}: the boundary controlled
electromagnetic waves (described by Maxwell's equations) in a bounded
domain containing a nonlinear conductor. \vspace{-1mm}

Our main result, stated and proved in Section \ref{sec6}, is that the
equations \rfb{Boris}, \rfb{Mirvis} determine a well-posed and
incrementally scattering passive nonlinear system. It is not obvious
how to define the latter concept and as a matter of fact it is one of
the aims of this paper to give a proper definition of a nonlinear
well-posed system, in Section \ref{sec6}. For this, we need to recall
the concept of Lax-Phillips semigroup of a well-posed linear system
$\Sigma$ and we do this in Section \ref{sec4}. We will then generalize
this concept to define Lax-Phillips type nonlinear semigroups in
Sections \ref{sec5} and \ref{sec6}. Section \ref{sec5} is devoted to
investigate classical and generalized solutions of \rfb{Boris}, while
Section \ref{sec6} deals with the well-posedness of
\rfb{Boris}-\rfb{Mirvis}. The results of Section \ref{sec5} are
reported also in \citep{ShWeTu:21}. \vspace{-1mm}

There are many papers dealing with maximal dissipative
nonlinear perturbations of maximal dissipative linear operators, to
show that the perturbed operator generates a (nonlinear) contraction
semigroup with certain stability properties, see for instance
\cite{Barbu,Brezis2,CranPazy, Marx}. Most references do not
consider inputs and outputs for the perturbed system. An important
novelty of our approach is that we prove the well-posedness of the
closed loop system, with inputs and outputs. Other papers that
consider the closed-loop system with inputs and outputs and use the
theory of monotone operators to prove its well-posedness are
\cite{Hastir_Hans} and \cite{Schmid}. The type of perturbations that they consider and their assumptions are rather different from ours and we were not able to unify our results with theirs i.e., to formulate a theory general enough to cover them all.

Very briefly, a {\em well-posed nonlinear system} $\Sigma^\NL$ is
determined by a strongly continuous (nonlinear) {\em Lax-Phillips
semigroup}. In particular for any input function $u\in\LEloc U$
(the space of $U$-valued functions that are square
integrable on any finite interval) and any initial state $x_0\in X$,
$\Sigma^\NL$ has a unique state trajectory in $C([0,\infty);X)$ and an
output function $y\in\LEloc Y$. Moreover, $x$ and $y$ depend
continuously on $x_0$ and $u$. Such a system $\Sigma^{\NL}$ is {\em
incrementally scattering passive} if its Lax-Phillips semigroup is
contractive. In this case, if $x_{01}$, $x_{02}$ are initial states in
$X$ and $u_1$, $u_2$ are input functions in $\LEloc U$, and we denote
by $x_1$, $x_2$ the corresponding state trajectories of $\Sigma^{\NL}$
and by $y_1$, $y_2$ the corresponding output functions of
$\Sigma^{\NL}$, then for all $\tau\geq 0$, the following inequality is
satisfied (similar to the energy balance inequality
\rfb{energy_balance}): \vspace{-2mm}
$$ \|x_1(\tau)-x_2(\tau)\|^2 + \int_0^\tau \|y_1(t)-y_2(t)\|^2 \dd t
   \ \ \ \ \ \m\vspace{-2mm}$$
\begin{equation} \label{energy_balance_M}
   \m\ \ \ \ \ \ \ \m\leq\m \|x_{01}-x_{02}\|^2 + \int_0^\tau 
   \|u_1(t)-u_2(t)\|^2 \dd t \m.
\end{equation}
In other words, for any $\tau\geq 0$, the (nonlinear) operator
$\Sigma_\tau^\NL$ which maps $(x(0),u)$ to $(x(\tau),y)$ is a
contraction. 

\section{Scattering passive linear systems} \label{sec2} 

We use the standard notations from functional analysis,
such as $\Lscr(X,Z)$ for bounded linear operators from $X$ to $Z$,
$\Dscr(A)$ for the domain of $A$, $\rho (A)$ for the resolvent set of
$A$.\vspace{-2mm}

We recall some background on well--posed linear systems,
following \citep{CuWe:89,Sala87,Staf_book,StWe02,TuWe_survey,Weiss10,
WeStTu01}. Let us denote by $U$ the input space, by $X$ the state
space and by $Y$ the output space of a well-posed linear system \m
$\Sigma$. $U$, $X$ and $Y$ are Hilbert spaces (only in this section,
we allow complex Hilbert spaces). The input and the output functions are
$u\in\LEloc U$ and $y\in\LEloc Y$ respectively. For any $u\in\LEloc U$
and any $\tau\geq 0$, we denote by $\PPP_\tau u$ its truncation to the
interval $[0,\tau]$. $\PPP_\tau u$ is regarded as a function in
$L^2([0,\infty);U)$, which is zero for $t>\tau$. We use the standard
notation $\Hscr^1((0,\infty);U)$ for the space of all $u\in
C([0,\infty);U)\cap L^2([0,\infty);U)$ for which there exists $\phi\in
L^2([0,\infty);U)$ such that \vspace{-2mm}
\begin{equation} \label{Trump_plan}
   u(t)-u(0) \m=\m \int^t_0\phi(\sigma) \dd\sigma \FORALL t\geq 0.
   \vspace{-2mm} 
\end{equation}
The space $\Hscr^1_\loc((0,\infty);U)$ consists of all $u\in C([0,
\infty);U)$ for which $\phi\in\LEloc U$ exists such that 
\rfb{Trump_plan} holds.\vspace{-1mm}

A well-posed linear system $\Sigma$ consists of the family of bounded
operators $\Sigma=(\Sigma_\tau)_{\tau\geq 0}$ such that \vspace{-2mm}
\begin{equation} \label{UNUZERO}
   \bbm{ x(\tau) \\ \PPP_\tau y } \m=\m \Sigma_\tau
   \bbm{ x(   0) \\ \PPP_\tau u } \m.\vspace{-2mm}
\end{equation}
Here $x:[0,\infty)\rarrow X$ is the state trajectory of $\Sigma$
corresponding to the initial state $x(0)$ and the input function $u$,
and $y$ is the corresponding output function. Denoting $c_\tau=\|
\Sigma_\tau \|$, \vspace{-2mm}
$$ \| x(\tau)\|^2 + \int_0^\tau \left\| y(t)\right\|^2 \dd t
   \m\leq\m c_\tau^2 \left( \| x(   0)\|^2 + \int_0^\tau \left\| 
   u(t)\right\|^2 \dd t \right).\nonumber \vspace{-3mm}$$
The operators $\Sigma_\tau$ are partitioned in a natural way
(corresponding to the two product spaces) as follows:\vspace{-2mm}
\begin{equation} \label{Sig4b}
   \Sigma_\tau \m=\m \bbm{\tline_\tau & \Phi_\tau\\
   \Psi_\tau & \fline_\tau} \m.\vspace{-3mm}
\end{equation}
The four families of operators appearing on the right-hand side above,
must satisfy four functional equations expressing the causality and
the time-invariance of \m $\Sigma$ (these functional equations are
parts of the definition of a well-posed system), see Section 2 of
\citep{StWe02}. In particular, the family $(\tline_\tau)_{\tau\geq 0}$
is a strongly continuous operator semigroup on $X$ and its generator
$A$ is called the {\em semigroup generator} of $\Sigma$. We introduce
$X_1=\Dscr(A)$ with the norm defined as $\|x\|_1=\|(\beta I-A)x\|$,
where $\beta\in \rho(A)$. $X_{-1}$ is the completion of $X$ with
respect to the norm $\|x\|_{-1}=\|(\beta I-A)^{-1}x\|$. These spaces
are independent of the choice of $\beta$, see \cite{obs_book}. $A$ has a unique extension
that is bounded from $X$ to $X_{-1}$, and we denote this extension by
the same symbol $A$. The semigroup $\tline$ can be extended to an
operator semigroup on $X_{-1}$, denoted by the same symbol, whose
generator is the extension of $A$ mentioned earlier. There exists a
unique operator $B\in \Lscr(U;X_{-1})$, called the {\em control
operator} of $\Sigma$ such that for all $t\geq 0$,\vspace{-2mm}
$$ \Phi_t u \m = \m \int^t_0 \tline_{t-\sigma} Bu(\sigma)\dd \sigma 
   \quad \forall u\in L^2([0,\infty);U).\vspace{-2mm}$$
The above integration is done in $X_{-1}$. There exists a unique {\em
observation operator} $C\in \Lscr(X_1,Y)$ so that for every
$\tau\geq 0$,\vspace{-2mm}
$$ (\Psi_\tau x_0)(t) \m=\m C\tline_t x_0 \quad \forall x_0\in
   \Dscr(A),\ t\in[0,\tau]. \vspace{-2mm}$$
$\Sigma$ can be fully described by two operators $\AB:\Dscr(\AB)
\rarrow X$, $\CD:\Dscr(\AB)\rarrow Y$ in a sense that will become
clear in Proposition \ref{YomKippur}. Here, \vspace{-2mm}
\begin{equation} \label{D(AB)}
   \Dscr(\AB) \m=\m \left\{\left.\nm\sbm{ x_{_0}\\[1mm] u_{_0}} \in 
   X\times U \ \right|\ Ax_0+Bu_0\in X \right\} \m,\vspace{-2mm}
\end{equation}
\begin{equation} \label{AB}
   [\AB]\sbm{x_0\\[1mm] u_0} \m=\m Ax_0+Bu_0 \m.
\end{equation}
Then $\Dscr(\AB)$ is a dense subspace of $X\times U$ and $\AB$ is
closed. Hence, the space $\Dscr(\AB)$ may be regarded as a Hilbert
space with the graph norm of $\AB$. $\CD$ is bounded from $\Dscr(\AB)$
to $Y$ and \vspace{-2mm}
\begin{equation} \label{CD_admissible}
   Cx \m=\m [\CD] \sbm{x\\[1mm] 0} \FORALL x\in\Dscr(A).\vspace{-2mm}
\end{equation}
In the particular case of a finite dimensional system described by the
equations $\dot x(t)=Ax(t)+Bu(t)$ and $y(t)=Cx(t)+D u(t)$, where
$A,B,C,D$ are matrices of suitable dimensions, we have $\AB=\left[A \
B\right]$, $\CD=\left[C \ D\right]$. The following proposition is
contained in \citep[Theorem 3.1]{StWe02}. \vspace{-1mm}

{\color{blue}
\begin{proposition} \label{YomKippur}
We use the notation introduced earlier in this section. Assume that $u
\in\Hscr^1_\loc((0,\infty);U)$ and $\sbm{x_0\\ u(0)}\in\Dscr(\AB)$. 
The state trajectory $x$ and the output function $y$ of $\Sigma$ are 
defined as in \rfb{UNUZERO}. Then \vspace{-3mm}
$$ x\in C^1([0,\infty);X), \quad \sbm{x\\[1mm] u}\in C([0,\infty);
   \Dscr(\AB)),\vspace{-2mm}$$
$$ y\in \Hscr^1_\loc((0,\infty);Y),$$
and for every $t\geq 0$ we have that\vspace{-2mm}
\begin{equation} \label{Kurds_abandoned}
   \dot{x}(t) \m=\m [\AB]\sbm{x(t)\\[1mm] u(t)}, \quad y(t) \m=\m 
   [\CD] \sbm{x(t)\\[1mm] u(t)}.\vspace{-2mm}
\end{equation}
\end{proposition}}

From the above it can be shown (by density and continuous extension)
that $\Sigma$ is completely determined by $\AB$ and $\CD$.
In particular, for $u\in \Hscr^1_\loc((0,\infty);U)$
with $u(0)=0$ we have $(\fline_\tau u)(t)=[\CD]\sbm{x(t)\\u(t)}$,
where $x(t)=\Phi_t u$. The observation operator of $\Sigma$, $C$ has
an extension $\bar{C}$ on $Z$, where the space $Z$ is defined as:
\vspace{-2mm}
\begin{equation} \label{SpaceZ}
   Z \m=\m \Dscr(A)+(\beta I-A)^{-1}BU.\vspace{-2mm}
\end{equation}
This is a Hilbert space with the norm \vspace{-3mm}
\begin{equation} \label{Chauvin_convicted}
   \|z\|^2_Z \m=\m \inf\left\{\ \nm\| x \|^2_1+\| v \|^2
   \left|\begin{array}{c} x \in X_1,\ \ v \in U,\\ z=x+(\beta I-A
   )^{-1}Bv \end{array}\right.\right\}.\vspace{-1.5mm}
\end{equation}
It must be noted that the extension $\bar{C}$ is not unique. For
each such extension $\bar{C}$, there exists $D\in\Lscr(U,Y)$ such
that \vspace{-2mm}
\begin{equation} \label{CD}
   [\CD]\sbm{x_0\\[1mm] u_0} =\m \bar{C}x_0+Du_0 \quad \forall 
   \sbm{x_0\\[1mm] u_0}\in\Dscr(\AB).\vspace{-4mm}
\end{equation}
 
{\em Scattering passive systems} are a subclass of the well-posed
linear systems, for which \rfb{energy_balance} holds, which is
equivalent to \vspace{-3mm}
\begin{equation} \label{ScaPassive}
   \frac{\dd}{\dd t}\|x(t)\|^2 \m\leq\m \|u(t)\|^2-\|y(t)\|^2,
   \vspace{-1mm}
\end{equation}
where $u,x$ and $y$ are as in Proposition \ref{YomKippur}. Thus,
$\Sigma$ is scattering passive if and only if $\|\Sigma_\tau\|\leq 1$
for all $\tau\geq 0$. \vspace{-1mm}

The following characterization of scattering passive systems has been
derived in \citep[Theorem 6.8]{TuWe_survey}. \vspace{-1mm}

{\color{blue}
\begin{proposition} \label{ScatteringPassive}
Let $A$ be the generator of a strongly continuous semigroup on $X$,
let $B\in \Lscr(U,X_{-1})$, let $\AB$ be defined as in
\rfb{D(AB)}-\rfb{AB}, and let $\CD$ be bounded from $\Dscr(\AB)$ to
$Y$. The operators $\AB$ and $\CD$ determine a scattering passive
well-posed linear system (via \rfb{Kurds_abandoned}) if and only if
\vspace{-2mm}
\begin{equation} \label{ScatteringCond}
   \Re\left\langle \bbm{ A\ \&\ B & 0 \\ \m 0\ -\half I & 0\\ C\ \&\ 
   D & \nm-\half I\m\m} \bbm{x_0\\ u_0\\ y_0}, \bbm{x_0\\ u_0\\ y_0} 
   \right\rangle \m\leq\m 0,
   \vspace{-2mm}
\end{equation}
for all $\left[x_0 \ u_0 \ y_0\right]^\top\in\Dscr(\AB)\times Y$ (the
inner product is computed in $X\times U\times Y)$.\vspace{-1mm}
\end{proposition}}

Note that according to the above proposition, \rfb{ScatteringCond}
implies that $B$ is an admissible control operator for the semigroup
$\tline$ generated by $A$ and the operator $C$ defined by
\rfb{CD_admissible} is an admissible observation operator for
$\tline$. For more background on admissible control and observation
operators we refer to \citep{HansenWeiss,JaPa_survey,Staf_book,
obs_book} and \citep{Weiss01}. Several other equivalent
characterizations of scattering passive systems can be found in
\citep[Proposition 7.2 and Theorem 7.4]{StWe02}. \vspace{-1mm}

In \citep{StWe12, WeSt13}, a class of scattering passive linear
systems with a special structure has been introduced. This class
appears in models of various systems from mathematical physics, such
as wave, plate and Maxwell equations (for the wave equation see
\citep{Kurula_Zwart}) and it contains the class of systems ``from thin
air'' introduced in \citep{TuWe,WeStTu01,Weiss-Tucsnak2003} as
explained in Remark 2.3 in \citep{StWe12}. Most scattering passive
systems of interest are in this class (if we take care to formulate
the state space and the operators in the right way). They are closely
related to the port-Hamiltonian systems analyzed in
\citep{JacobSkrepek,JacZwa_12,Skrepek}. Following the survey paper
\citep{TuWe_survey}, this class is called the ``Maxwell class'' of
systems. The precise connection between the class of
scattering passive port-Hamiltonian boundary control systems and the
class of systems in the Maxwell class is complicated, neither contains
the other. There are several important examples of systems that belong
to both. We briefly recall the main facts about the Maxwell class of systems. \vspace{-1mm}

If $\Sigma$ is a system in the {\em Maxwell class}, then its state
space $X$ can be decomposed as $X=H\oplus E$, where $H$ and $E$ are
Hilbert spaces. The Hilbert space $U$ is both the input space and the
output space of $\Sigma$. We identify $H$, $E$ and $U$ with their
duals $H'$, $E'$ and $U'$. The Hilbert space $E_0$ is a dense subspace
of $E$ and the embedding $E_0\hookrightarrow E$ is continuous. We
denote by $E_0'$ the dual of $E_0$ with respect to the pivot space
$E$, so that \vspace{-3mm}
\begin{equation} \label{independence_day}
   E_0\subset E \subset E_0' \m,\vspace{-3mm}
\end{equation}
densely and with continuous embeddings. We decompose the state of \m 
$\Sigma$ as follows: \m $x_{_0}=\sbm{z_{_0}\\[1mm] q_{_0}}$\m,
$z_{_0}\in H$, $q_{_0}\in E$\m. The following theorem is extracted
from results in \citep{StWe12}. \vspace{-1mm}

{\color{blue}
\begin{theorem} \label{MainPThm}
Let $H,E,U$ and $E_0$ be as in the previous paragraph, and let the 
operators $L\in\Lscr(E_0,H)$, $K\in\Lscr(E_0,U)$ and \m
$G\in\Lscr(E_0,E_0')$ be such that \vspace{-2mm}
\begin{equation} \label{A4_new}
   \sbm{L\\[1mm] K} \colon E\to H\oplus U \text{ (with
   domain $E_0$) is closed},\vspace{-2mm}
\end{equation}
\begin{equation} \label{A2_new}
   \Re \langle Gq_{_0},q_{_0} \rangle_{E_0',E_0} \m\leq\m 0 \FORALL
   q_{_0}\in E_0 \m.
\end{equation} \label{AB_Maxwell}
The operators $A,B,\bar{C},D$ are defined as follows: \vspace{-2mm}
\begin{equation} \label{SPDef}
   A \m=\m \bbm{0 & -L\\ L^* & G -{\textstyle\half} K^* K},\quad 
   B \m=\m \bbm{0\\K^*}\vspace{-2mm}
\end{equation} \vspace{-2mm}
\begin{equation} \label{DomainA}
   \n\nm\m\Dscr(A) = \left\{\nm\bbm{z_{_0}\\ q_{_0}}\in\nm 
   \begin{array}{c} H\\[-2mm] \times\\[-2mm] E_0\end{array} \right|
   \,\left. \bigbluff L^* z_{_0} + \bigl( G - {\textstyle\half} K^*
   K \bigr) q_{_0} \in E \right\}.
\end{equation}
Accordingly, $\AB$ is defined by \rfb{D(AB)} and \rfb{AB}. The 
operator $\CD$ is defined by \rfb{CD}, where\vspace{-2mm}
$$ \bar{C} \m=\m \left[0 \ -K\right],\qquad \Dscr(\bar{C}) \m=\m 
   E_0,\qquad D \m=\m I.\vspace{-2mm}$$
Then $\AB$ and $\CD$ determine a scattering passive system $\Sigma$.
Moreover, the following claims hold:\vspace{-3mm}
\begin{itemize}
\item[1] If the input function $u$ and the initial state $\sbm{z(0)
\\ q(0)}$ of \m $\Sigma$ satisfy \vspace{-6mm}
\begin{equation} \label{TRoosevelt}
   u \m\in\m \Hscr^1_\loc((0,\infty);U) \m,\qquad \sbm{z(0)\\[1mm] 
   q(0)\\[1mm] u(0)} \m\in\m \Dscr(\AB) \m,\vspace{-2mm}
\end{equation}
then the corresponding state trajectory $\sbm{z\\ q}$ and output
function $y$ of $\Sigma$ satisfy \vspace{-4mm}
\begin{equation} \label{EHoover}
   \sbm{z\\[1.5mm] q} \m\in\m C^1([0,\infty);H\oplus E) \m,\quad
   \sbm{z\\[1mm] q\\[1mm] u}\in C([0,\infty);\Dscr(\AB)) \m,
   \vspace{-2mm}
\end{equation} 
$$y \in \Hscr^1_\loc((0,\infty);Y) \m,$$
and the system is represented by \rfb{Kurds_abandoned}.
\smallskip
\item[2] The semigroup generator $A$ is the restriction of the
   operator \vspace{-2mm}
\begin{equation} \label{A3PR}
   \overline{A} \m=\m \bbm{0 & -L \\ L^* & G-{\textstyle\half}
   K^* K}
\end{equation}
(defined on $H\times E_0$, with values in $H\times E_0'$) to the
domain $\Dscr(A)$ from \rfb{DomainA}.
\smallskip
\item[3] We denote by $X_1$ the space $\Dscr(A)$ with the norm
$\|z\|_1=\|(I-A)z\|$ and by $X_{-1}$ the completion of $X$
with respect to the norm $\|z\|_{-1}=\|(I-A)^{-1}z\|$. We have
\vspace{-2mm}
$$ X_1 \m\subset\m H \times E_0 \m\subset\m X \m\subset\m
   H \times E_0' \m\subset\m X_{-1} \m, \vspace{-2mm}$$
densely and with continuous embeddings. $A$ has a unique extension
to an operator $A\in\Lscr(X,X_{-1})$, whose restriction to $H
\times E_0$ is $\overline{A}$ from \rfb{A3PR}.
\smallskip
\item[4] If $u$, $x=\sbm{z\\ q}$ and $y$ are as in
\rfb{TRoosevelt}--\rfb{EHoover}, then they satisfy the following
{\em power balance equation} for every $t\geq 0$:\vspace{-2mm}
\begin{equation}\label{pow_bal}
\n\n\frac{\dd}{\dd t} \|x(t)\|^2 = \|u(t)\|^2-
   \|y(t)\|^2 + 2\Re\ipd{G q(t)}{q(t)}. \m\vspace{-2mm}
\end{equation}
\end{itemize}
\end{theorem}}

We remark that the additional claim 1 is a consequence of the main
statement together with Proposition \ref{YomKippur}, while claim 2
follows the main statement together with \rfb{D(AB)} and \rfb{AB}.
\vspace{-1mm}

The class of systems described in the above theorem are called the 
Maxwell class. Thus, any system in this class is described by the
following equations (for all $t\geq 0$):\vspace{-2mm}
\begin{equation} \label{Tutti_left}
   \begin{aligned} \bbm{\dot z(t) \\ \dot q(t)} &= \bbm{0 & -L\\ 
   L^* & G - \half K^* K}\bbm{z(t)\\ q(t)} + \bbm{0 \\ K^*}
   u(t),\\ y(t) &= \bbm{0 & -K} \bbm{z(t)\\ q(t)} + u(t). 
   \end{aligned}\vspace{-2mm}
\end{equation}
It is clear from \rfb{A2_new} and \rfb{pow_bal} that systems in this
class are scattering passive. \vspace{-2mm}

We introduce also the class of {\em impedance passive systems
in the Maxwell class}, as in equation (1.8) in
\citep{StWe12} (with slightly changed notation). These systems are
described by the following equations: For $t\geq 0$, \vspace{-2mm}
\begin{equation} \label{Tiszapart}
   \begin{aligned} \bbm{\dot z(t) \\ \dot q(t)}
   &= \bbm{0 & -L\\ L^* & G}\bbm{z(t) \\ q(t)} + \bbm{0 \\ K_0^*} 
   e(t), \\ f(t) &= \bbm{0 & \phantom{-}K_0} \bbm{z(t) \\ q(t)},
   \qquad K_0 \m=\m \frac{1}{\sqrt{2}} K \m.\end{aligned}\vspace{-1mm}
\end{equation}
Here the state space is again $X=H\oplus E$, the input signal is $e$
(also called {\em control effort}), the output signal is $f$ (called
{\em flow}), and the assumptions on the operators $L,K,G$ are exactly
as in Theorem \ref{MainPThm}. These systems are system nodes in the
sense of \citep{MaStWe,StWe12,TuWe_survey}. The corresponding operator
semigroup generator is $A_\imp$, which is the restriction of
\vspace{-3mm}
\begin{equation} \label{Eminem}
   \overline{A_\imp} \m=\m \bbm{0 & -L \\ L^* & G}\vspace{-2mm}
\end{equation}
(defined on $H\times E_0$, with values in $H\times E_0'$) to the
domain \vspace{-3mm}
\begin{equation} \label{Pietro_with_Covid}
   \Dscr(A_\imp) \m=\m \left\{\bbm{z_{_0}\\ q_{_0}}\in 
   \begin{array}{c} H\\[-2mm] \times\\[-2mm] E_0\end{array}
   \right| \left. \bigbluff L^* z_{_0} + G q_{_0} \in E \right\}.
   \vspace{-2mm}
\end{equation}
The control operator is $B_\imp=\sbm{0\\ K_0^*}$ and the operator
$[\CD]_\imp$ is defined similarly as in \rfb{CD}, with $\overline C
_\imp$ in place of $\overline C$, and 0 in place of $D$, where 
$\overline C_\imp=[0\ \ K_0]$. \vspace{-1mm}

The systems in this class need not be well-posed. They have state and
output trajectories for a dense set of initial states and input
functions, like any system node, and these satisfy \rfb{Tiszapart}.
Moreover, these state and output trajectories satisfy\vspace{-1mm}
\begin{equation} \label{Hannsa_vaccinated1}
   \half\frac{\dd}{\dd t} (\|z(t)\|^2+\|q(t)\|^2) \m=\m \Re
   \left\langle Gq(t),q(t)\right\rangle+\Re\left\langle e(t),f(t)
   \right\rangle \m.
\end{equation}
Thanks to \rfb{A2_new} we obtain that\vspace{-2mm}
\begin{equation} \label{Hannsa_vaccinated}
   \frac{\dd}{\dd t} (\|z(t)\|^2+\|q(t)\|^2) \m\leq\m
   2\Re\left\langle e(t),f(t) \right\rangle \m.\vspace{-1mm}
\end{equation}
This estimate \rfb{Hannsa_vaccinated} is what defines the class of
{\em impedance passive systems}. They include passive electric
circuits where $e$ is the applied voltage (vector) and $f$ is the
current. \vspace{-1mm}

For systems as in \rfb{Tiszapart}, we can introduce new input and
output signals via the {\em external Cayley transformation}:
\vspace{-3mm}
$$ u \m=\m \frac{1}{\sqrt 2}(e+f),\qquad y \m=\m \frac{1}{\sqrt
   2}(e-f),\vspace{-1mm}\vspace{-1mm}$$
which can be interpreted also as a feedback loop and a feedforward
loop around the original system, see \citep[Fig.~1]{StWe12}. After this
transformation, our system is described by \rfb{Tutti_left} and it is
scattering passive, which means that \rfb{ScaPassive}
holds. \vspace{-1mm}

The main result outlined in Section \ref{sec1} (the precise statement
is in Section \ref{sec6}) allows us to generalize this class of
systems by allowing the appearance of a maximal monotone nonlinear
operator $\Nscr$ subtracted from $G$ in \rfb{SPDef}. Thus, the
operator $\overline A$ is replaced with the nonlinear operator
\vspace{-2mm}
\begin{equation} \label{A_nonl}
   \overline{A} \m=\m \bbm{0 & -L\\ L^* & G-\Nscr-{\textstyle\half}
   K^* K} \m,\vspace{-2mm}
\end{equation}
where $\Nscr:E\rarrow E$ is maximal monotone. The well-posedness of
this particular class of nonlinear systems is the topic of our
conference paper \citep{ShWeTu:20}. Unfortunately, there are some
mistakes in the example in \citep{ShWeTu:20}, the corrected version is
here in the next section. \vspace{-1mm}

\section{Motivating examples} \label{sec3} 

In the modelling of physical systems, we often come across second 
order differential equations with a nonlinear damping term depending
on the velocity, such as \vspace{-3mm}
\begin{eqnarray}\label{Alstom_to_be_sold}
   \left\{\begin{array}{ll}\ddot x(t) + D \dot x(t) + \Nscr(\dot x(t))
   + A_0 x(t) \m\ni\m B_0 e(t) \m, \\ f(t) \m=\m C_0 \dot x(t) + D_0 
   e(t) \m.\end{array}\right.\vspace{-2mm}
\end{eqnarray}
Here $x(t)\in E$, where $E$ is a finite-dimensional inner product
space. The function $x\in C^2([0,\infty);E)$ usually represents a
vector of displacements, $A_0,\ D\in\Lscr(E)$ are such that $A_0>0$ 
and $D\geq 0$ and $\Nscr:E\rarrow E$ is a monotone set-valued 
function. The $U$-valued signals $e$ and $f$ are the input and the 
output of the system ($U$ is another finite-dimensional space), 
while $B_0\in\Lscr(U,E)$, $C_0\in\Lscr(E,U)$ and $D_0\in\Lscr(U)$.
\vspace{-1mm}

The addition of the damping term $\Nscr$ may create a much more
complex dynamic behaviour (compared to the linear case), as the
following example illustrates. \vspace{-1mm}

\begin{example} \label{3.1} {\rm
Let us assume that $x$ represents the one-dimensional displacement of
a rigid body with mass $m>0$ along a straight line, under the
influence of the external force $e$, while it is connected to the
point denoted by zero on this straight line via a spring with constant
$k>0$, having viscous friction with coefficient $d>0$, $A_0=k/m$,
$D=d/m$ and $B_0=1/m$. Suppose that the nonlinear function $\Nscr$ is
\vspace{-2mm}
\begin{equation} \label{Nigeria_schoolgirls}
   \Nscr(v) \m=\m \beta \m {\rm sign}(v) \m,\vspace{-2mm}
\end{equation}
where $m\beta>0$ is the amplitude of the Coulomb (or static) friction
force and sign (the multi-valued signum function) is defined by
\vspace{-5mm}
\begin{equation} \label{Noya}
   {\rm sign}(v) \m=\m \left\{ \begin{array}{ll} 1 &\mbox{ if } v>0 
   \m,\\[-1mm] -1 &\mbox{ if } v<0 \m,\\[-1mm] {\rm [-1,1]}
   &\mbox{ if } v=0 \m.\end{array} \right. \vspace{-2mm}
\end{equation}
It is well known that in this case, for any initial state $(x(0),\
\dot{x}(0))$ and any continuous function $e$, \rfb{Alstom_to_be_sold}
has a unique solution. If $e$ is sufficiently small, then this
solution stops in finite time (at a point that may depend on $e$ and
on $(x(0),\ \dot{x}(0))$), see for instance \cite{AmDi}, \cite{DiMi}.
Thus, the system has a continuum of equilibrium points, none of which
is locally asymptotically stable. If we replace $\Nscr(v)$ with the
shifted version $\Nscr(v-v_0)$ (which is still monotone, and
represents Coulomb friction with respect to a moving platform having
velocity $v_0\not=0$) then the system has a globally asymptotically
stable equilibrium point.}
\end{example} \vspace{-1mm}

The second example (below) will illustrate that the addition of the
damping term $\Nscr$ in \rfb{Alstom_to_be_sold} does not necessarily
improve the stability properties of the system. We should not
necessarily think of $\Nscr$ as a clever addition to the system, meant
to improve it, but as some physical phenomenon present in the system
that needs to be modeled. \vspace{-1mm}

\begin{example} \label{3.2} {\rm
Consider a mechanical system consisting of two rigid bodies with
masses $m_1$ and $m_2$, two springs with constants $k_1$ and $k_2$ and
a damper with constant $d$, all moving along a straight line, as in
Fig. \ref{MassSpringDamper}. Its equations are \vspace{-2mm}
\begin{eqnarray}
   \left\{\begin{array}{ll} m_1 \ddot x_1(t)+k_1 x_1(t)+k_2
   (x_1(t)-x_2(t)) \m=\m 0 \m,\\ m_2 \ddot x_2(t)-k_2(x_1(t)-x_2(t)) 
   + d \dot x_2(t) \m=\m e(t) \m,\\ f(t) \m=\m \dot x_2(t).
   \end{array}\right.\nonumber\vspace{-2mm}
\end{eqnarray}
Here $x_1$ and $x_2$ denote the displacements of the rigid bodies from
their equilibrium positions and $e$ is an external force acting on the
second rigid body. Assuming that all the constants are positive, it is
easy to check that this system is exponentially stable. If, in
addition, there is also a static friction force of amplitude $\beta$
between the second body and its supporting surface, then the
second equation changes to \vspace{-2mm}
$$ m_2 \ddot x_2(t)-k_2(x_1(t)-x_2(t)) + d \dot x_2(t) + \beta \m {\rm
   sign}(\dot x_2(t)) \m\ni\m e(t) \m,\vspace{-1mm}$$
where the multi-valued function ${\rm sign}$ is as in \rfb{Noya}. This
system can be put into the framework \rfb{Alstom_to_be_sold} with
$\zeta=\sbm{\zeta_1\\ \zeta_2}=\sbm{\sqrt{m_1}x_1\\ \sqrt{m_2}x_2}$ in
place of $x=\sbm{x_1\\x_2}$, by defining $E=\rline^2$, $U=\rline$,
$C_0=\bbm{0 & \frac{1}{\sqrt{m_2}}}$, $B_0=C_0^\top$,
$$ A_0 \m=\m \bbm{ \frac{k_1+k_2}{m_1} & -\frac{k_2}{\sqrt{m_1m_2}}\\
   -\frac{k_2}{\sqrt{m_1m_2}} &\frac{k_2}{m_2} },\qquad D \m=\m 
   \bbm{0 & 0\\ 0 & \frac{d}{m_2}} ,$$
$$ \Nscr\left(\bbm{\zeta_1\\ \zeta_2}\right) \m=\m \bbm{0\\ 
   \frac{\beta}{\sqrt{m_2}} \m {\rm sign}(\zeta_2)} \m,\qquad 
   D_0 \m=\m 0 \m.$$
Note that $A_0>0$ and $D\geq 0$, as assumed after
\rfb{Alstom_to_be_sold}. If we introduce new state variables
$z=-A_0^\half\zeta$ and $q=\dot{\zeta}$, then the linear version of
the system fits into the framework of \rfb{Tiszapart}, with
$L=L^*=A_0^\half$, $G=-D$ and $K_0=C_0$. \vspace{-1mm}

\begin{figure}[t] 
   \centering 
   \includegraphics[height=2cm, width=6cm]{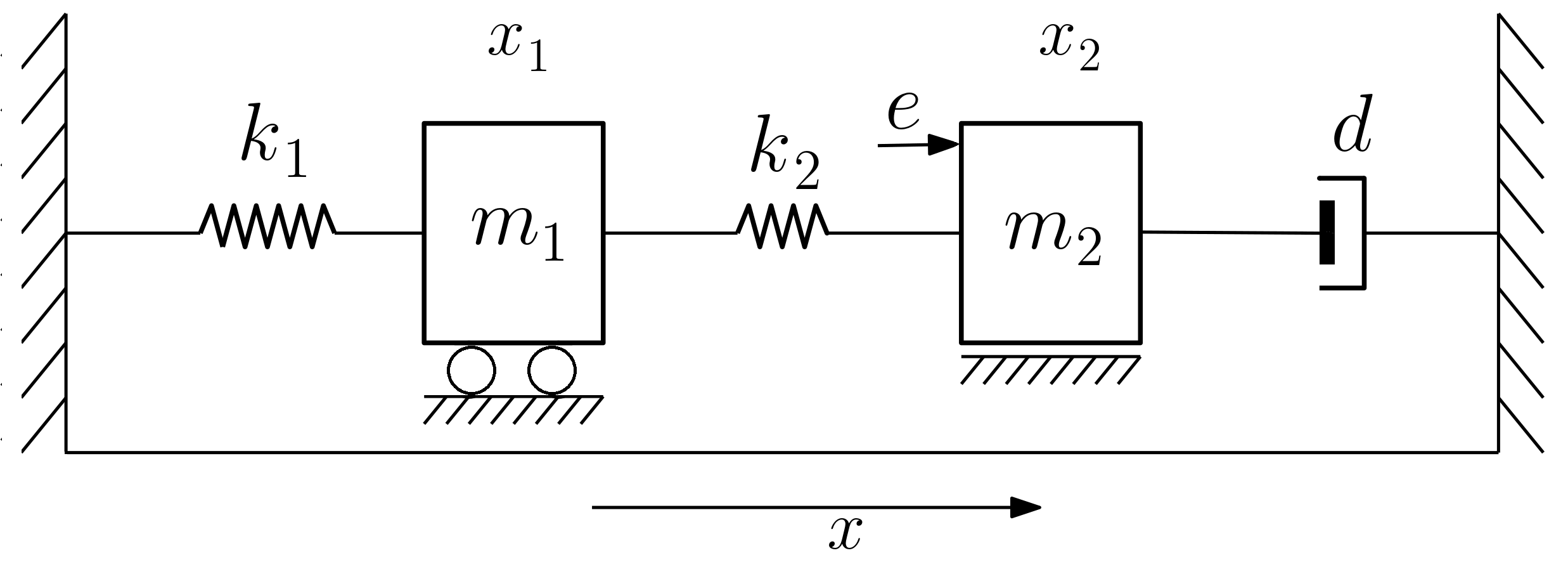}
   \vspace{-2mm}
   \caption{The mechanical system from Example \ref{3.2}. Rigid bodies
   with masses $m_1$ and $m_2$ are connected via two springs (with
   spring constants $k_1$ and $k_2$) and a damper (with constant $d$).
   The external force acting on $m_2$ is $e$.}\label{MassSpringDamper}
\end{figure}

The addition of static friction does not improve the stability of the
system, quite the contrary. Assuming that $e=0$, a typical state
trajectory is such that $\dot x_2(t)$ becomes zero in finite time,
after which $x_1(t)$ oscillates with constant amplitude. \vspace{-1mm}

This paper is about systems that include infinite-dimensional versions
of \rfb{Alstom_to_be_sold}, and our main interest is their
well-posedness, considering input and output signals of class $L^2$.

There is a large literature on systems described by linear partial
differential equations with a nonlinear damping term, acting in the
interior or on the boundary of the domain. We mention, as a
representative sample only, the papers \citep{AlabAmma, Barbu}
(abstract second order in time equations), \citep{Berr} (beam
equations), \citep{chitour, ConRao, Haraux} (wave equations),
\citep{ELN_decay} (Maxwell's equations), \citep{Lasi_1989} (wave and
plate equations), \cite{LaTa} (wave equations), \cite{Zuazua_1990}
(wave equations). As far as we are aware, most of the papers on this
topic treat the well-posedness of the associated Cauchy problem, i.e.,
the existence of a non-linear strongly continuous semigroup on the
state space, and its various asymptotic stability properties (decay
rates). It seems that little attention has been devoted to systems
with input and output signals that are described by equations
containing a nonlinear damping term.  Our aim in this paper is to fill
this gap in an abstract and fairly general framework.}
\end{example} \vspace{-1mm}

\begin{example} \label{SCOLE_ex} {\rm
Consider the vibrations in a fixed vertical plane of a vertical beam
clamped at the bottom, having a rigid body with a large mass $M$
mounted on the top. Such a system could represent, for instance, a
wind turbine tower with the nacelle and the turbine together playing
the role of the rigid body as shown in Fig. \ref{tower}. If we adopt
the homogeneous Euler-Bernoulli model for the beam, then this is the
famous SCOLE system, introduced in \cite{LittMa:881,LittMa:882} (the
authors had in mind an antenna on a flexible mast). Suppose that a
perturbation force $F$ acting horizontally on the rigid body in the
fixed vertical plane, causes the beam to vibrate. In the case of the
wind turbine, this force would represent the wind acting on the
turbine and the nacelle. The well-posedness and other properties of
this system (linear version) were analyzed in \cite{ZhWe:10}, with
many more relevant references. \vspace{-1mm}

We try to dampen the vibrations of this system by placing a trolley of
mass $m$ in contact with the rigid body, with friction between
them. This friction has a component of viscous friction with constant
$D$ and a component of static friction with amplitude $F_0$. The idea
is to absorb the vibration energy via these frictions. Such dampers or
more sophisticated versions, called {\em tuned mass dampers}, are
often used to dampen the vibrations of very tall buildings, see, e.g.,
\cite{TMD_Fitz}, \cite{Hro}, \cite{Var} and the references
therein. Assuming that the beam is uniform, with height $l$, the model
(the SCOLE system coupled with the trolley, defined for $(x,t)\in
(0,l)\times [0,\infty)$), is the following collection of equations:
\vspace{-2mm}
\begin{equation} \label{eq:SCOLE}
   \n\left\{\begin {array}{llll} \rho w_{tt}(x,t)+ EI w_{xxxx}(x,t)
   \m=\m 0,\qquad\\ w(0,t) \m=\m 0\m,\qquad\ w_{x}(0,t) \m=\m 0 \m,\\ 
   M w_{tt}(l,t)-EIw_{xxx}(l,t) \m=\m F(t)-D [w_t(l,t)-\xi_t(t)]\\ 
   \hspace{12mm}-F_0 {\rm sign} [w_t(l,t)-\xi_t(t)] - k [w(l,t)-
   \xi(t)] \m, \\ J{w}_{xtt}(l,t)+ EI w_{xx}(l,t) \;=\> 0 \m, \\ m 
   \xi_{tt}(t) \m=\m D [w_t(l,t)-\xi_t(t)]\\ \hspace{12mm} +F_0 {\rm 
   sign} [w_t(l,t)-\xi_t(t)] + k [w(l,t)-\xi(t)],\end{array} \right.
\end{equation}
where the subscripts $t$ and $x$ denote derivatives with respect to
the time $t$ and the position $x$, respectively. We have denoted by
$w$ the transverse displacement of the beam, and by $\xi$ the
horizontal displacement of the trolley with respect to an equilibrium
position. The positive constants $EI$, $\rho$ and $J$ are the flexural
rigidity of the beam, the mass density of the beam and the moment of
inertia of the rigid body. We have assumed that the trolley is
connected to the rigid body by a spring with constant $k$, whose role
is to prevent the trolley from drifting away too far from its
equilibrium position. The function sign has been defined in
\rfb{Noya}. The linear version of this system (corresponding to
$F_0=0$) but with a non-uniform beam, has been investigated in Section
5 of \cite{ZhWe:17}. \vspace{-1mm}

The signal $F$ is the force input acting on the rigid body.
$-EIw_{xxxx}(x,t)\dd x$ is the total lateral force acting on a slice
of the beam of length $\dd x$, located at the position $x$ and the
time $t$. $EIw_{xxx}(l,t)$ and -$EIw_{xx}(l,t)$ are the force and the
torque acting on the rigid body from the beam at $t$. The input and
output signals of system \rfb{eq:SCOLE}, $e$ and $f$, are as follows:
\vspace{-2mm}
\begin{equation} \label{eq:SCOLEoutput}
   e \m=\m F \m,\qquad f \m=\m  w_{t}(l,\cdot) \m.\vspace{-2mm}
\end{equation}
A convenient choice for the state of this nonlinear coupled system is
$\sbm{z\\ q}$, where \vspace{-2mm}
$$ z(t) = \bbm{\sqrt{\rho}\m w_t(\cdot,t)\\ \sqrt{k}(\xi(t)- w(l,t))
   \\ \sqrt{J}\m w_{xt}(l,t)} \m,\quad
   q(t) = \bbm{\sqrt{EI}\m w_{xx}(\cdot,t)\\ \sqrt{M}\m w_t(l,t)\\ 
   \sqrt{m}\m \xi_t(t)} \m.$$
We denote by $z_1(t),\m z_2(t),\m z_3(t)$ the components of $z(t)$,
and similarly for the vector $q(t)$. The state space is
\vspace{-2mm}
$$ X \m=\m L^2[0,l]\times\rline^2 \times L^2[0,l]\times\rline^2 \m,
   \vspace{-2mm}$$
with the natural product norm (here $L^2[0,l]$ consists of real-valued
functions). The physical energy in the system is 
$\half\|z\|^2+\frac{1}{2}\|q\|^2=\half\|[z \quad q]^\top\|^2_X$.
\vspace{-1mm}

\begin{figure}[h!] 
   \centering 
   \includegraphics[height=9cm, width=5cm]{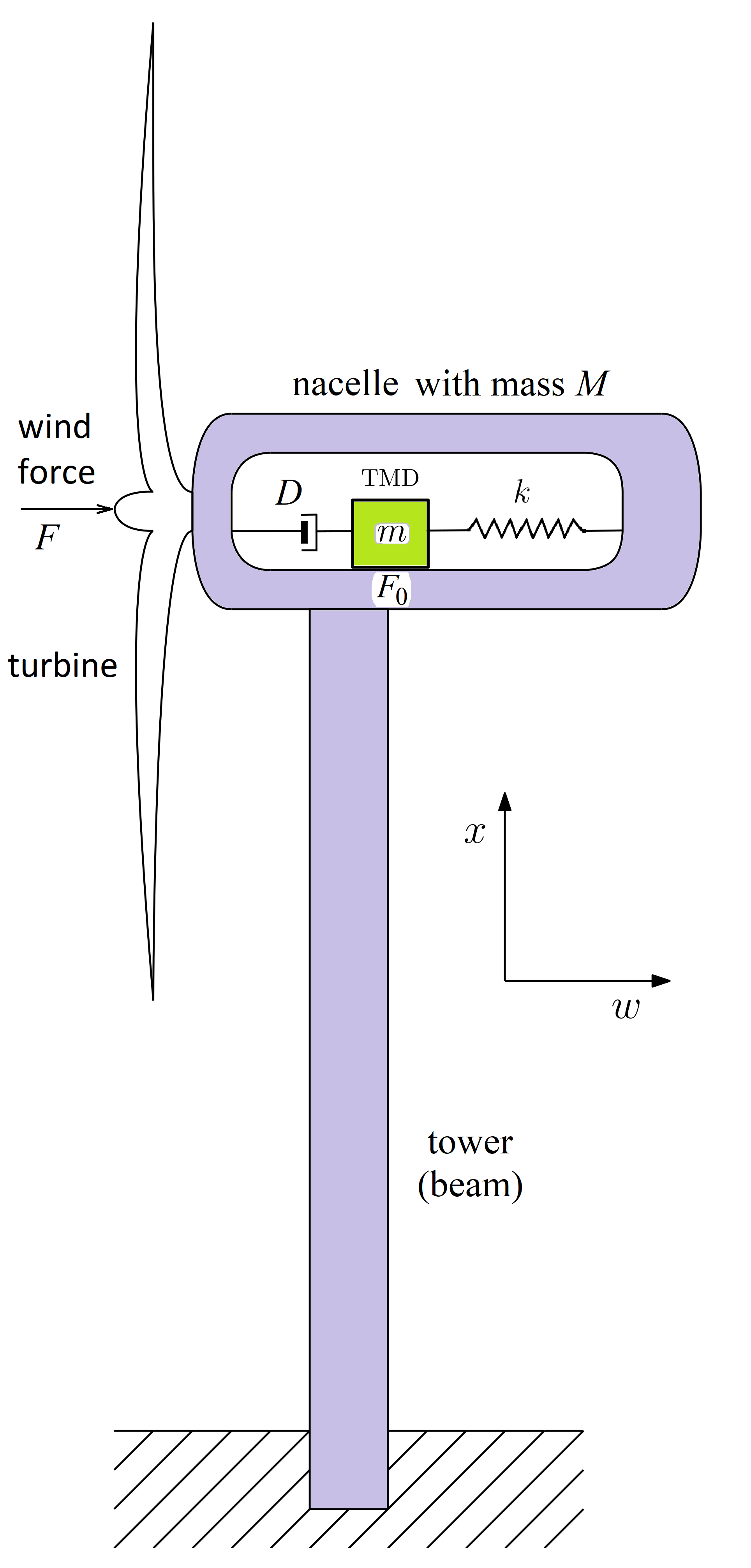}
   \caption{A wind tower coupled with a tuned mass damper (TMD). The
   tower is clamped at the bottom with a heavy mass $M$ (nacelle
   together with turbine) on the top and it is described by the SCOLE
   model. The mass of the trolley (i.e., of the TMD) is $m$ and it is connected to the
   nacelle via a spring and a damper.} \label{tower}
\end{figure}

The {\em linear version} of this system, which corresponds to $F_0=0$,
can be expressed as an impedance passive system in the Maxwell class,
as in \rfb{Tiszapart}. To show this, we define \vspace{-2mm}
$$ H \m=\m E \m=\m L^2[0,l]\times\rline^2 \m,\quad E_0 \m=\m 
   \Hscr^2(0,l)\times\rline^2\m,\vspace{-1mm}$$
and $U=\rline$. For any $\zeta\in [0,l]$, $\delta_\zeta\in
(\Hscr^2(0,l))'$ denotes the unit pulse at $\zeta$, so that its
adjoint $\delta^*_\zeta$ is the operator of point evaluation at
$\zeta$ defined on $\Hscr^2(0,l)\hspace{-0.5mm}:\delta^*_\zeta
\varphi=\varphi(\zeta)$. The operators $L$ and $K_0$ from
\rfb{Tiszapart} are defined by \vspace{-2mm}
$$ L \m=\m \bbm{\scriptstyle\sqrt{\frac{EI}{\rho}}\frac{\dd^2}{\dd 
   x^2} & 0 & 0 \\ 0 & \scriptstyle \sqrt{\frac{k}{M}} & \scriptstyle
   -\sqrt{\frac{k}{m}}\\ \scriptstyle \sqrt{\frac{EI}{J}}\delta^*_l &
   0 & 0} \m,\quad K_0 \m=\m \bbm{0 & \scriptstyle\frac{1}{\sqrt{M}} 
   & 0}.$$
The operator $G\in\Lscr(E_0,E_0')$ from \rfb{Tiszapart} is defined by
\vspace{-3mm}
$$ G \m=\m \bbm{0 & \scriptstyle\sqrt{\frac{EI}{M}}\m \delta'_l & 0\\
   \scriptstyle \sqrt{\frac{EI}{M}}\delta_l^* \frac{\dd}{\dd x} & 
   \scriptstyle -\frac{D}{M} & \scriptstyle \frac{D}{\sqrt{Mm}} \\ 
   0 & \scriptstyle \frac{D}{\sqrt{Mm}} & \scriptstyle 
   -\frac{D}{m}},$$
and the adjoint of $G$ is computed to be \vspace{-3mm}
$$ G^* \m=\m \bbm{0 & -\scriptstyle\sqrt{\frac{EI}{M}}\m \delta'_l & 
   0\\ -\scriptstyle \sqrt{\frac{EI}{M}}\delta_l^* \frac{\dd}{\dd x}
   & \scriptstyle -\frac{D}{M} & \scriptstyle \frac{D}{\sqrt{Mm}} \\ 
   0 & \scriptstyle \frac{D}{\sqrt{Mm}} & \scriptstyle 
   -\frac{D}{m}}.$$
Notice that $G+G^*\in\Lscr(E)$ and $G+G^*\leq 0$, so that \rfb{A2_new}
holds. \vspace{-1mm}

It is a bit tricky to compute $L^*$. For any $\zeta\in[0,l]$, we
denote by $\delta'_\zeta$ the derivative of $\delta_\zeta$, then its
adjoint ${\delta'_\zeta}^*$ is the functional ${\delta'_\zeta}^*
\varphi=\langle\varphi,\delta'_\zeta\rangle=-\varphi'(\zeta)$, for any
$\varphi\in\Hscr^2(0,l)$. We have that $\frac{\dd^2}{\dd x^2}\in
\Lscr(\Hscr^2(0,l);L^2(0,l))$ so that $\left[\frac{\dd^2}{\dd x^2}
\right]^*\in \Lscr(L^2(0,l);(\Hscr^2(0,l))')$, however we only compute
its restriction to $\Hscr^2(0,l)$. It follows from integration by
parts that the restriction of $\left[\frac{\dd^2}{\dd x^2}\right]^*$
to $\Hscr^2(0,l)$ is given by \vspace{-2mm}
$$ \left[\frac{\dd^2}{\dd x^2}\right]^* = \frac{\dd^2}{\dd x^2}+
   \Delta, \ \ \  \Delta = \delta_l{\delta'_l}^*-\delta'_l\delta_l^*
   -\delta_0{\delta'_0}^*+\delta'_0\delta_0^*.$$
Therefore, the restriction of $L^*$ to $E_0$ is the operator matrix 
\vspace{-2mm}
$$ L^* \m=\m \bbm{\scriptstyle\sqrt{\frac{EI}{\rho}}\left(\frac{\dd^2}
   {\dd x^2}+\Delta\right) & 0 & \scriptstyle \sqrt{\frac{EI}{J}}
   \delta_l \\ 0 & \scriptstyle \sqrt{\frac{k}{M}} & 0\\ 0 & 
   \scriptstyle -\sqrt{\frac{k}{m}} & 0}.\vspace{-2mm}$$
The domain of the corresponding semigroup generator $A_\imp$, computed
according to \rfb{Pietro_with_Covid}, turns out to be \vspace{-2mm}
$$ \Dscr(A_\imp) \m=\m \left\{ \bbigbluff\bbm{z_{_0}\\ q_{_0}}\in 
   \begin{array}{c} \vspace{-1mm} E_0\\ \times\\ E_0\end{array} \right|
   \ \left. \bluff \begin{array}{c} \sqrt{J}z'_1(l)=\sqrt{\rho}z_3\\
   \sqrt{M}z_1(l)=\sqrt{\rho}q_2\\ z_1(0)=0\\ z'_1(0)=0 \end{array}
   \right\}.\vspace{-4mm}$$

Now we can verify by computations that \rfb{Tiszapart} is equivalent
to the linear version of \rfb{eq:SCOLE}. This linear system is {\em
impedance passive} in the sense of \rfb{Hannsa_vaccinated}, a
consequence of the theory presented in Section \ref{sec2}.
\vspace{-1mm}

In the nonlinear case, $G$ in \rfb{Tiszapart} and
\rfb{Hannsa_vaccinated1} has to be replaced with $G-\Nscr$, where
$\Nscr$ is a (set-valued) maximal monotone operator defined by
\vspace{-2mm}
$$ \Nscr(q) \m=\m \bbm{ 0\\ \frac{F_0}{\sqrt{M}} {\rm sign} 
   \left( \frac{1}{\sqrt{M}} q_2 - \frac{1}{\sqrt{m}} q_3\right)\\
   -\frac{F_0}{\sqrt{m}} {\rm sign} \left( \frac{1}{\sqrt{M}} q_2 - 
   \frac{1}{\sqrt{m}} q_3\right)} \m.$$
The definition of $\Nscr$ is such that if it happens
that the argument of sign, which we denote by \vspace{-3mm}
$$ Q \m=\m \frac{1}{\sqrt{M}} q_2 - \frac{1}{\sqrt{m}} q_3\m,
   \vspace{-1mm}$$
is zero, then the selection of sign\m $Q$ in the second and third row
of $\Nscr(q)$ must be the same number from $[-1,1]$. 
This is because of Newton's law of action and reaction: the static 
friction force from the TMD to the nacelle must be opposite to the 
static friction force from the nacelle to the TMD. Then it is easy to
check that \vspace{-1mm}
$$ \langle \Nscr(q)-\Nscr(\tilde q),q-\tilde q\rangle \m=\m
   F_0 \left( {\rm sign}\m Q-{\rm sign}\m \tilde Q\right) 
   \left( Q-\tilde Q\right) \geq 0 \m,\vspace{-1mm}$$
so that indeed $\Nscr$ is monotone. By an easy argument that we 
omit, $\Nscr$ has no proper monotone extension. \vspace{-1mm}

The equations \rfb{Tiszapart} and the estimate
\rfb{Hannsa_vaccinated1}, with $G-\Nscr$ in place of $G$, hold for any
solution of \rfb{eq:SCOLE}. The fact that such solutions exists
follows from Theorem \ref{Gen_sol} and Remark \ref{bounded_B} later in
this article. After applying the external Cayley transformation to
system \rfb{eq:SCOLE}, the transformed system is described by
\rfb{Tutti_left}, but again with $G-\Nscr$ in place of $G$. This
nonlinear system is scattering passive (in particular, it is
well-posed), as follows from the theory in Section \ref{sec6}.}
\end{example} \vspace{-1mm}

\section{Some background about the Lax-Phillips semigroup} 
\label{sec4} 

Starting from an arbitrary well-posed linear system $\Sigma$, it is
possible to define a strongly continuous semigroup which resembles
those encountered in the scattering theory of \cite{LaPhBook,LaPh73},
and which contains all the information about $\Sigma$. We recall the
basics about this semigroup, following \cite{StWe02} (related material
is also in \cite{ChWe:05} and \cite{Staf_book}).\vspace{-1mm}

Like in the previous section, we assume that $\Sigma$ is a well-posed
linear system with component operator families as in \rfb{Sig4b}, and
we continue to use also the notation $U$, $X$ and $Y$. For any $\tau
\geq 0$, we denote by $\SSS_\tau$ the (unilateral) right shift
operator by $\tau$ on $\Uscr=L^2([0,\infty);U)$ and also on
$\Yscr=L^2((-\infty,0];Y)$, so that their adjoints $\SSS^*_\tau$ are
the operators of left shift by $\tau$ on the same spaces. We also
introduce $\Sscr_t$, the bilateral right shift by $t$ acting on
$L^2((-\infty,\infty);Y)$ (where $t\in\rline$). We regard $\Yscr$ as a
subspace of $L^2((-\infty,\infty);Y)$, by extending functions in
$\Yscr$ to be zero for $t>0$.\vspace{-1mm}

{\color{blue}
\begin{proposition} \label{LaxPhilProp}
For all $t\geq 0$ we define on $\Yscr\times X\times\Uscr$ the 
operator $\ULax_t$ by\vspace{-5mm}
$$ \ULax_t \m=\m \bbm{\Sscr_{-t} & 0            & 0 \\
                          0      & I            & 0 \\
                          0      & 0            & \SSS_t^*}
                 \bbm{    I      & \Psi_t       & \fline_t\\
                          0      & \tline_t     & \Phi_t\\
                          0      & 0            & I}.\vspace{-2mm}$$
Then $\ULax=(\ULax_t)_{t\geq 0}$ is a strongly continuous semigroup.
\end{proposition}}

If we take $y_0\in\Yscr$, $x_0\in X$ and $u_0\in\Uscr$ to represent
the past output function of $\Sigma$ (for $t<0$), its initial state
and its input, then at any time $t\geq 0$, the first component of
$\ULax_t\sbm{y_0 \\ x_0 \\ u_0}$ represents the past output up to
$t$, the second component represents the present state $x(t)$ and the
third component represents the future input that will reach $\Sigma$
after $t$. The operator semigroup $\ULax$ introduced in the last
proposition is called the {\em Lax-Phillips semigroup} of
$\Sigma$. Translating scattering theory into the language of systems
theory was pioneered in \cite{Helt76}. The generator of $\ULax$ can
be characterized as follows: \vspace{-1mm}

{\color{blue}
\begin{proposition} \label{LaxPhillipsGenTh}
Let $\ULax$ be the Lax--Phillips semigroup of the well-posed system
$\Sigma$. We denote the generator of $\ULax$ by $\GothA$, and we use
the notation $\AB$ and $\CD$ from \rfb{Kurds_abandoned}.\vspace{-2mm}

\begin{enumerate}
\item The domain of $\GothA$, $\Dom\GothA$ consists of all the vectors
$\sbm{y_0\\ x_0\\ u_0}\in\Hscr^1((-\infty,0);Y)\times X\times\Hscr^1
((0,\infty);U)$ which satisfy $\sbm{x_0\\ u_0(0)}\in\Dscr(\AB)$ and 
$y_0(0)=\CD\sbm{x_0\\ u_0(0)}$, and on $\Dom\GothA$, $\GothA$ is 
given by \vspace{-4mm}
\begin{equation} \label{GothAact}
   \GothA \bbm{y_0 \\ x_0 \\ u_0} \m=\m
   \bbm{ \vspace{1mm} y_0' \\ \vspace{1mm} \AB \sbm{x_0\\ u_0(0)} \\ 
   u_0'}. \vspace{-3mm}
\end{equation}
\item The following two conditions are equivalent: \vspace{-1mm}

\m\ \ {\rm (a)} $\sbm{y_0 \\ x_0 \\ u_0}\in\Dom{\GothA}$ and
$\sbm{y\\ x\\ u} \m=\m \GothA \sbm{y_0 \\ x_0 \\ u_0}$,

\m\ \ {\rm (b)} $y_0\in\Hscr^1((-\infty,0);Y)$, $x_0\in X$,\\ $u_0
\in\Hscr^1((0,\infty);U)$, $\sbm{x_0\\ u_0(0)}\in\Dscr(\AB)$ and
\vspace{-2mm}
\begin{equation} \label{5*}
   \bbm{x \\y_0(0)} = \bbm{\AB\\ \CD} \bbm{x_0 \\ u_0(0)} ,
   \quad \bbm{y \\ u} = \bbm{y_0' \\ u_0'} \m.
\end{equation}
\end{enumerate}
\end{proposition}}

This proposition has been extracted (as a particular case) from 
Theorem 6.3 in \cite{StWe02}. \vspace{-1mm} 

\begin{remark} \label{GothA_equivalent} \rm
The operator $\GothA$ from \rfb{GothAact} can be written also in the
form \vspace{-3mm}
\begin{equation} \label{GothA_LP}
   \GothA \m=\m \bbm{\left[\frac{\dd}{\dd \xi}\right]_\Yscr & 
   \delta_0\bar{C} & \delta_0 D\delta_0^*\\ 0 & A & B\delta_0^*\\ 0 
   & 0 & \left[\frac{\dd}{\dd \xi}\right]_\Uscr}, \vspace{-1mm}
\end{equation} 
where we have used the splitting of $C\&D$ as in \rfb{CD}, and where
$\left[\frac{\dd}{\dd\xi}\right]_\Yscr$ is the generator of the left
shift semigroup on $\Yscr$, whose domain is $\Yscr_1=\Hscr^1_0(
(-\infty,0);Y)$, and $\left[\frac{\dd}{\dd\xi}\right]_\Uscr$ is the
generator of the left shift semigroup on $\Uscr$, whose domain is
$\Uscr_1=\Hscr^1((0,\infty);U)$. Thus, for any $\varphi\in\Hscr^1
((-\infty,0);Y)$, $\left[\frac{\dd}{\dd\xi}\right]_\Yscr\varphi$ is
in $\Yscr_{-1}$, the dual space of $\Hscr^1((-\infty,0);Y)$:
\vspace{-4mm}
$$ \m\ \ \ \ \ \ \left[\frac{\dd}{ \dd \xi}\right]_\Yscr \varphi 
   \m=\m \varphi' - \delta_0\varphi(0) \m.\vspace{-2mm}$$
This is similar to the discussion in Example 4.2.7 in 
\citep{obs_book}.
\end{remark}

{\color{blue}
\begin{proposition} \label{attack_in_Halle}
We use the notation of the previous two propositions. The following
conditions are equivalent:\vspace{-2mm}
\begin{enumerate}
\item $\Sigma$ is scattering passive.
\item The Lax-Phillips semigroup induced by $\Sigma$ is contractive
(equivalently, $\|\ULax_t\|=1$ for all $t\geq 0$).  
\end{enumerate}
\end{proposition}}

This proposition has been extracted from Proposition 7.2 in
\cite{StWe02}. The fact that $\|\ULax_t\|\leq 1$ is equivalent to
$\|\ULax_t\|=1$ follows from the structure of $\ULax_t$: it contains
blocks that are left shifts, hence $\|\ULax_t\|$ cannot be less than
1.\vspace{-1mm}

\section{Classical and generalized solutions of an abstract 
         differential inclusion} \label{sec5} 

First we recall very briefly some facts about strongly continuous
semigroups of nonlinear operators. For the basics about such
semigroups we refer to \cite{Brezis74, CranPazy, Kato_accr, Show}.
\vspace{-1mm}

Let $Z$ be a real Hilbert space. A {\em strongly continuous semigroup
of nonlinear operators} $\ULax$ acting on $Z$ is defined exactly as in
the linear case, without requiring that the operators are linear.  If
$\ULax$ is such a semigroup, then define the operator \vspace{-2mm}
\begin{equation} \label{operator_minimal}
   \GothA^0 z \m=\m \lim_{t\rarrow 0,\m t>0} \frac{1}{t} \left[ 
   \ULax_t z - z\right] \m, \end{equation}
   \begin{equation}\label{D(operator_minimal)} \Dscr(\GothA^0) \m=\m
   \left\{ z\in Z\ |\ \mbox{the above limit exists} \right\} \m.
   \end{equation} 
Following \cite{CranPazy}, $\GothA^0$ is called the {\em (strong)
generator} of $\ULax$. Very little is known about semigroups of
nonlinear operators at this level of generality. However, there is a
rich body of knowledge about a subclass of such semigroups, those
that are contractive. The semigroup $\ULax$ is called {\em
contractive} if\vspace{-1mm}
$$ \|\ULax_t z_1 - \ULax_t z_2\| \m\leq\m \|z_1-z_2\| 
   \n\n\FORALL z_1,z_2\in Z, \m\m t\geq 0.\vspace{-3mm}$$

{\color{blue}
\begin{theorem} \label{Putin}
Assume that $\ULax$ is contractive. Then the operator $\GothA^0$ from
\rfb{operator_minimal} and \rfb{D(operator_minimal)} is densely
defined and dissipative. This operator $\GothA^0$ has a unique maximal
dissipative extension $\GothA$ (which might be set-valued) with the
same domain $\Dscr(\GothA)=\Dscr(\GothA^0)$. If $z_0\in
\Dscr(\GothA)$, then $\GothA^0z_0$ is the unique element of smallest
norm in the closed and convex set $\GothA z_0$. \vspace{-1mm}

Let $z_0\in\Dscr(\GothA)$. The function $z:[0,\infty)\rarrow Z$
defined by $z(t)=\ULax_tz_0$ is Lipschitz continuous and right
differentiable at every $t\geq 0$. Moreover, for every $t\geq 0$,
it holds that $z(t)\in \Dscr(\GothA)$,\vspace{-2mm}
\begin{equation} \label{d^+z/dt}
   \frac{\dd^+z(t)}{\dd t} \m=\m \GothA^0 z(t),\vspace{-2mm}
\end{equation}
and $\GothA^0 z$ is right continuous at $t$.  
\end{theorem}}\vspace{-1mm}

In \rfb{d^+z/dt} $\frac{\dd^+}{\dd t}$ denotes the right derivative.
By $\GothA^0$ being densely defined we mean of course that $\Dscr
(\GothA^0)$ is dense in $Z$. The above theorem follows from Theorems
1.3, 1.5 and A1 as well as Corollary 3.1 in \cite{CranPazy}. In the
linear case, $\GothA^0=\GothA$. \vspace{-1mm}

{\color{blue}
\begin{theorem}[Crandall-Pazy] \label{Crandall-Pazy}
Let $\GothA$ be a maximal dissipative set-valued operator on $Z$ with
domain $\Dscr(\GothA)$ dense in $Z$. For each $z_0\in\Dscr(\GothA)$
let $\GothA^0z_0$ denote the element of smallest norm in $\GothA
z_0$. Then there is a unique strongly continuous semigroup of
nonlinear operators $\ULax$ acting on $Z$ such that $\GothA^0$ is the
generator of $\ULax$. Moreover, $\ULax$ is contractive.
\end{theorem} } \vspace{-1mm}

This theorem is due to \cite{CranPazy}, see their Theorem
\RomanNumeralCaps{1}. It is a generalization of the well known
Lumer-Phillips theorem from linear semigroup theory (see for instance
\cite{Engel_Nagel} or \cite{obs_book}).  A related result has been
proved in \cite{CisVel} for time varying nonlinear systems in which
the nonlinear operator is continuously Fr\'echet differentiable. 
\vspace{-1mm}

{\color{blue}
\begin{theorem}[Kato] \label{Kato}
Consider the differential inclusion\vspace{-2mm}
\begin{equation}\label{Kato_theo}
   \frac{\dd^+\psi}{\dd t}(t) \m\in\m \GothA \psi(t)+\gamma \psi(t),
   \vspace{-2mm}
\end{equation}
where $\GothA$ is a maximal dissipative operator on the Hilbert space
$Z$ and $\gamma\geq 0$. For each $\psi_0\in\Dscr(\GothA)$, there is a
unique absolutely continuous $\psi:[0,\infty)\rarrow Z$, such that
$\psi(0)=\psi_0$, $\psi$ is Lipschitz continuous on any finite time
interval, it is right-differentiable, $\psi(t)\in\Dscr(\GothA)$ and
\rfb{Kato_theo} holds.
\end{theorem}} \vspace{-1mm}

For a proof see Theorem 4.1 in \cite{Show}.\vspace{-1mm}

\begin{definition} \label{NonlinearSemigroupDef}
Let $\GothE :\Dscr(\GothE)\rarrow Z$ be a set valued
operator with $\Dscr(\GothE)$ dense in $Z$. We say that
$\GothE$ {\em determines a strongly continuous semigroup} of
nonlinear operators on $Z$ if there exists a unique such semigroup
$\ULax$ such that: \vspace{-1mm}
\begin{enumerate}[(a)]
\item for each $\psi_0\in\Dscr(\GothE)$, the function
$\psi(t)=\ULax_t\psi_0$ is {\em Lipschitz continuous} on any 
finite time interval and right differentiable for all $t\geq 0$.
\item The above function $\psi$ is such that $\psi(t)\in\Dscr
(\GothE)$ for all $t\geq 0$ and \vspace{-5mm}
\begin{equation} \label{NonlinearGenerator}
   \frac{\dd^+\psi}{\dd t}(t) \m\in\m \GothE\psi(t) \quad 
   \forall t\geq 0.\vspace{-1mm}
\end{equation}
\end{enumerate}
\end{definition}

\begin{corollary} \label{C_0_nonlinearSemi}
Let $Z$, $\GothA$ and $\gamma$ be as in Theorem \ref{Kato} with
$\Dscr(\GothA)$ dense in $Z$. Then the operator $\GothA+\gamma I$
determines a strongly continuous semigroup of nonlinear operators
$\ULax$ on $Z$. Moreover for any $\psi_0,\phi_0\in \Dscr(\GothA)$,
\begin{equation} \label{Zelensky}
   \|\ULax_t\psi_0-\ULax_t\phi_0\| \m\leq\m 
   e^{\gamma t}\|\psi_0-\phi_0\| \quad \forall t\geq 0.\vspace{-2mm}
\end{equation}
\end{corollary} \vspace{-1mm}

{\em Proof.} \m For $\psi_0\in\Dscr(\GothA)$ we define $\ULax_t
\psi_0=\psi(t)$, where $\psi$ is the function from Theorem \ref{Kato}.
It is easy to derive from \rfb{Kato_theo} that \rfb{Zelensky} holds,
which shows that $\ULax_t\psi_0$ depends continuously on
$\psi_0$. Hence, by density and continuous extension we can define
$\ULax_t\psi_0$ for any $\psi_0\in Z$. The semigroup property
and the strong continuity of $\ULax_t$ are easy to
prove. \ $\square$ \vspace{-1mm}

\begin{definition} \label{AbstractSystem}
Let $X$ and $U$ be Hilbert spaces and $\Uscr=L^2([0,\infty);U)$. An
{\em abstract nonlinear control system} with input space $U$ and state
space $X$ is a family $\Sigma^{\rm st}=(\Sigma^{\rm st}_{\tau})_{\tau
\geq 0}$ of continuous operators from $X\times\Uscr$ to $X$, which
satisfies the following for any $x_0\in X$ and any $u\in\Uscr$:
\vspace{-1mm}\begin{enumerate}[\rm (a)]
\item {\em Composition property:} \vspace{-3mm}
$$ \Sigma^{\rm st}_{t+\tau} \sbm{x_0\\[2mm] u} \m=\m \Sigma^{\rm st}_t
   \bbm{\Sigma^{\rm st}_\tau \sbm{x_0\\[1mm] u} \\ \SSS_\tau^* u}\ \ 
   \forall \ t, \tau\geq 0. \vspace{-1mm}$$
\item {\em Identity property:} $\Sigma^{\rm st}_0\sbm{x_0\\[1mm] u}
   \m=\m x_0$.
\item {\em Continuity property:} $x(t)=\Sigma^{\rm st}_t\sbm{x_0
   \\[1mm] u}$ is a continuous function of \m $t\geq 0$.
\smallskip
\item {\em Causality property:} recalling the notation $\PPP_\tau$ 
for operators of truncation as in Section \ref{sec2}, \vspace{-2mm}
$$ \Sigma^{\rm st}_\tau \sbm{x_0 \\[2mm] \PPP_\tau u} \m=\m 
   \Sigma^{\rm st}_\tau \sbm{x_0 \\[2mm] u} \FORALL \tau\geq 0.$$
\end{enumerate} 
\end{definition} \vspace{-1mm}

Note that property (d) implies that the value of the state
$x(\tau)=\Sigma^{\rm st}_\tau\sbm{x_0\\ u}$ does not depend on the
future values (meaning for $t>\tau$) of the input signal $u$.
Properties (a) and (d) above can be compressed into one as
follows: \vspace{-2mm}
\begin{equation} \label{comp+causal}
   \Sigma^{\rm st}_{t+\tau}\sbm{x_0 \\[2mm] \displaystyle u
   \ICON_\tau v} \m=\m \Sigma^{\rm st}_t\bbm{\Sigma^{\rm st}_\tau
   \sbm{x_0\\[1mm] u}\\ v}\ \ \ \ \forall\ t,\tau\geq 0. 
   \vspace{-2mm}
\end{equation}
Here we have used the usual notation for the $\tau$-{\em 
concatenation} of two functions $u,v\in\Uscr$: \vspace{-3mm}
\begin{equation}\label{concat}
   ({\displaystyle u\ICON_\tau v})(t) \m=\m \left\{\begin{array}{c} 
   u(t)\\v(t-\tau)\end{array} \n \begin{array}{c} \ \ \ \m 
   \text{for}\quad t\in [0,\tau),\\ \text{for}\quad  t\geq \tau.
   \end{array}\right. \vspace{-1mm}
\end{equation} 
It is clear that \rfb{comp+causal} implies the composition property,
by taking $v=\SSS_\tau^* u$, and also the causality property, by 
taking $t=0$ and $v=0$. It is easy to see that (a) and (d) imply
\rfb{comp+causal}. \vspace{-1mm}

Our definition above is a special case of the one given in
\cite[Definition 1.3]{Mironchenko}, where the space of control inputs
is not specified to be $L^2$, it merely has to be a normed space on
which left shifts are contractions and concatenations are
allowed. Moreover, the definition in \cite{Mironchenko} allows for
blow-up of the state trajectories in finite time. \vspace{-1mm}

For the remainder of this section we study the differential inclusion
\rfb{Boris}. We must clarify what we mean by classical and generalized
solutions of \rfb{Boris}. As in Section \ref{sec4}, $U$ and $X$ will
denote real Hilbert spaces and $\Uscr=L^2([0,\infty);U)$. 
\vspace{-2mm}

\begin{definition} \label{Clas_Sol}
Assume that $A$ is the generator of the strongly continuous (linear)
semigroup $\tline$ on $X$ with domain $\Dscr(A)$, $\Mscr$ is a maximal
monotone (set-valued) operator defined on all of $X$ and
$B\in\Lscr(U;X_{-1})$. \vspace{-5mm}

\m\ \ The pair $(x,u)$ is called a {\em classical solution} of 
\rfb{Boris} if \vspace{-2mm}
\begin{enumerate}[\rm (a)]
\item $x\in C([0,\infty);X)$ and $x$ is right differentiable for all
      $t\geq 0$,
\item $u\in\Hscr^1((0,\infty);U)$,
\item $Ax(t)+Bu(t)\in X$ for all $t\geq 0$ (which implies 
      $x(t)\in Z$),
\item the inclusion \rfb{Boris} holds for all $t\geq 0$.\\
The pair $(x,u)$ is called a {\em generalized solution} of 
\rfb{Boris} if 
\item $x\in C([0,\infty);X)$,
\item $u\in L^2([0,\infty);U)$,
\item there exists a sequence $(x_n,u_n)$ of classical solutions of
\rfb{Boris}, such that $x_n(t)\rarrow x(t)$ in $X$ for all $t\geq 0$,
and $u_n\rarrow u$ in $L^2([0,\infty);U)$.
\end{enumerate}
\end{definition}\vspace{-1mm}

The following theorem gives a sufficient condition for \rfb{Boris}
to have classical and generalized solutions. \vspace{-1mm}

{\color{blue}
\begin{theorem} \label{Gen_sol}
With the assumptions and the notation of Definition \ref{Clas_Sol},
assume that there exists $\l\geq 0$ and $\beta>0$ such that for any
$\sbm{x_0\\u_0}\in\Dscr(\AB)$ (as defined in \rfb{D(AB)}),
\vspace{-2mm}
\begin{equation} \label{Lax-Phillips_like}
   \left\langle\bbm{A-\l I & B \\ 0 & -\beta I} \bbm{x_0\\u_0},
   \bbm{x_0\\u_0}\right\rangle \leq\m 0.\vspace{-2mm}
\end{equation}
Then for any $u\in\Hscr^1((0,\infty);U)$ and any $x_0\in X$ such that
$Ax_0+Bu(0)\in X$, there exists a unique function $x:[0,\infty)
\rarrow X$ such that $(x,u)$ is a classical solution of \rfb{Boris}
and $x(0)=x_0$. These classical solutions determine a unique abstract
nonlinear control system $\Sigma^{\rm st}$ with the input space $U$
and the state space $X$, in the sense that (for $x_0$ and $u$ as
above) \vspace{-3mm}
$$ x(t) \m=\m \Sigma^{\rm st}_t \sbm{x_0\\[1mm] u} \FORALL t\geq 0.
   \vspace{-2mm}$$
Let $x_0,z_0\in X$, $u,v\in\Hscr^1((0,\infty);U)$, let $(x,u)$ be a
classical solution of \rfb{Boris} with $x(0)=x_0$ and let $(z,v)$ be
a classical solution of \rfb{Boris} with $z(0)=z_0$. Then for all
$t\geq 0$, \vspace{-2mm}
\begin{equation} \label{Sol_inequ}
   \|x(t)-z(t)\|^2 \m\leq\m \hspace{48mm} \m\vspace{-3mm}
\end{equation} \vspace{-3mm}
$$ \m\ \ \ \ e^{2\l t} \left[\|x_0-z_0\|^2 +  2\beta\int^t_0 e^{-2\l 
   \sigma}\|u(\sigma)-v(\sigma)\|^2 \dd\sigma\right].$$
\end{theorem}}

\begin{remark} \label{equi_cond1_re} \rm
The inequality \rfb{Lax-Phillips_like} is clearly equivalent to
\vspace{-2mm}
\begin{equation} \label{equi_cond1}
   \langle Ax_0+Bu_0,x_0\rangle \m\leq\m \l\|x_0\|^2+\beta
   \|u_0\|^2 \quad \forall \sbm{x_0\\[1mm] u_0}\in \Dscr(A\&B).
\end{equation}
A more sophisticated equivalent estimate is \rfb{Equiv_cond}.
\end{remark}

\begin{remark} \rm The estimate \rfb{Sol_inequ} shows the continuity
of the classical solutions of \rfb{Boris} with respect to the initial
state and the input function. 
\end{remark}

{\em Proof of Theorem \ref{Gen_sol}.} \m Let us define
a new inner product on $U$ by \vspace{-1mm}
$$\langle u, v\rangle_{new} \m=\m 2\beta\langle u, v\rangle,$$
with the corresponding norm $\|u\|_{new}$. Then 
\rfb{Lax-Phillips_like} becomes (in terms of the new inner product on
$X\times U$) \vspace{-2mm}
$$ \left\langle \bbm{A-\l I & B\\ 0 & -\half I} \bbm{x_0\\ v_0},
   \bbm{x_0\\ v_0} \right\rangle_{\n\m new} \m\leq\m 0.
   \vspace{-2mm}$$
This implies that the system $\Sigma$ determined by $\AB-[\l\ 0]$ and
$Y=\{0\}$ (hence $\CD=0$) is scattering passive with state space $X$
and input space $U$ (with the new norm), see Proposition
\ref{ScatteringPassive}. The state trajectories of $\Sigma$ are
described by \ $\dot{z}(t)=(A-\l I)z(t)+Bv(t)$, for all $t\geq 0$,
where $v\in\Uscr=L^2([0,\infty);U)$. Recall from Section \ref{sec3} 
that $\delta_0^*$ denotes the operator of point evaluation at $0$.
The generator $\Ascr$ of the Lax-Phillips semigroup of the
system $\Sigma$ is \vspace{-2mm}
\begin{equation} \label{linearOp_A}
   \Ascr \m=\m \bbm{A-\l I & B\delta_0^* \\0 & \frac{\dd}{\dd\xi}}.
   \vspace{-2mm}
\end{equation} 
This follows from \rfb{GothA_LP} with $Y=\{0\}$. The domain of $\Ascr$
is \vspace{-2mm}
$$ \Dscr(\Ascr) \m=\m \left\{ \sbm{z_0\\ v_0}\in X\times\Hscr^1((0,
   \infty);U) \ |\ Az_0 + Bv_0(0)\in X \right\} \vspace{-2mm}$$
(see Proposition \ref{LaxPhillipsGenTh}). The operator $\Ascr$ is 
maximal dissipative and densely defined on $X\times\Uscr$, according
to Proposition \ref{attack_in_Halle} and Theorem \ref{Putin} (the
linear version). \vspace{-1mm}

On replacing $A$ with $A-\Mscr$ and $\frac{\dd}{\dd \xi}$ with
$\frac{\dd}{\dd\xi}-\lambda I$ in \rfb{linearOp_A}, where $\Mscr$ is a
nonlinear (possibly multi-valued) maximal montone operator defined on
all of $X$, we obtain the nonlinear operator \vspace{-4mm}
\begin{equation}\label{split_Ascr}
   \Ascr^\Mscr \m=\m \underbrace{\bbm{A-\l I & B\delta_0^* \\0 & 
   \frac{\dd}{\dd\xi}}}_{\text{$\Ascr$}} - \underbrace{\bbm{\Mscr & 
   0\\ 0 & \l I}}_{\text{$\tilde{\Mscr}$}}, \vspace{-2mm}
\end{equation} 
with $\Dscr(\Ascr^\Mscr)=\Dscr(\Ascr)$. Clearly $\tilde{\Mscr}=\sbm
{\Mscr & 0\\ 0 & \l I}\vspace{1mm}$ is a maximal monotone operator
defined on all of $X\times\Uscr$, therefore $\Dscr(\Ascr)\cap\left(
{\rm int}\Dscr(\tilde{\Mscr})\right)=\Dscr(\Ascr)$, which is dense.
According to Theorem 1 of \cite{Rocka}, it follows that $\Ascr^\Mscr$
is {\em maximal dissipative and densely defined} on $X\times\Uscr$. It
follows from Corollary \ref{C_0_nonlinearSemi} that the operator
$\Ascr_0=\Ascr^\Mscr+\l I$ determines strongly continuous semigroup
$\ULax$ on $X\times\Uscr$. Clearly \vspace{-2mm}
$$ \Ascr_0 \m=\m \bbm{A-\Mscr & B\delta_0^* \\ 0 & \frac{\dd}{\dd 
   \xi}}, \quad \Dscr(\Ascr_0)=\Dscr(\Ascr).\vspace{-2mm}$$
Denoting the first component operator of $\ULax_\tau$ by
$\Sigma^{\rm st}_\tau$, we have \vspace{-2mm}
\begin{equation} \label{Lax-Phillip_Semi}
   \ULax_\tau \m=\m \bbm{\ \  \ \Sigma^{\rm st}_\tau\\ 0 & \n
   \SSS_\tau^*} \quad \forall \tau\geq 0.\vspace{-2mm}
\end{equation}
It follows from the semigroup properties that $\Sigma^{\rm st}_\tau$
is continuous (from $X\times \Uscr$ to $X$) and the
composition, identity and continuity properties (from Definition
\ref{AbstractSystem}) hold.

Now we prove the causality of $\Sigma^{\rm st}$. We choose $\tau>0$,
fixed during this step of the proof. Denote, for all $t\geq 0$,
\vspace{-2mm}
\begin{equation} \label{z(t)Lax-Phillips like}
   \bbm{z(t)\\ v_t} \m=\m \ULax_t\bbm{z_0\\ v},\qquad \bbm{\tilde{z}
   (t)\\ \tilde{v}_t} \m=\m \ULax_t\bbm{z_0\\ {\displaystyle v
   \ICON_{\tau} 0}}, \vspace{-2mm}
\end{equation}
where $\sbm{z_0\\[1mm] v}\in\Dscr(\Ascr)$ and $v(\tau)=0$, such that
we have $\sbm{z_0\\[1mm]{\displaystyle v\ICON_{\tau}0}}\in\Dscr
(\Ascr)$. According to Definition \ref{NonlinearSemigroupDef} and
Corollary \ref{C_0_nonlinearSemi}, the two $X\times\Uscr$-valued
functions defined in \rfb{z(t)Lax-Phillips like} are Lipschitz
continuous and right differentiable for all $t\geq 0$, they stay in
$\Dscr(\Ascr)$ for all $t\geq 0$ and satisfy \vspace{-1mm}
$$ \frac{\dd^+}{\dd t}\sbm{z(t)\\[1mm] v_t} \m\in\m \Ascr_0\sbm{z(t)
   \\[1mm] v_t},\qquad \frac{\dd^+}{\dd t} \sbm{\tilde{z}(t)\\[1mm] 
   \tilde{v}_t} \m\in\m \Ascr_0\sbm{\tilde{z}(t)\\[1mm] 
   \tilde{v}_t}.$$
Hence, the functions $z(t)$ and $\tilde{z}(t)$ (emanating from the
same initial value $z_0$) satisfy the following differential
inclusions: \vspace{-1mm}
\begin{equation} \label{Trj1}
   \dot{z}(t) \m\in\m (A-\Mscr)z(t) + B\delta_0^* v_t,
\end{equation}
\begin{equation} \label{Trj2}
   \dot{\tilde{z}}(t) \m\in\m (A-\Mscr)\tilde{z}(t) + B\delta_0^*
   \tilde{v}_t. \vspace{1mm}
\end{equation}
Recall that ${\displaystyle v\ICON_\tau 0}=\PPP_\tau v$. Using the
fact that $v_t=\SSS^*_tv$ and $\tilde{v}_t=\SSS^*_t\PPP_\tau v$, we
get that for all $t\geq 0$, \vspace{-2mm}
$$ \n\n\n\n\frac{\dd^+}{\dd t}z(t) \m\in\m (A-\Mscr)z(t) + Bv(t),
   \vspace{-2mm}$$
$$ \frac{\dd^+}{\dd t}\tilde{z}(t)\m\in\m (A-\Mscr)\tilde{z}(t) +
   B(\PPP_\tau v)(t). \vspace{1mm}$$
Taking the difference and taking inner product with
$z(t)-\tilde{z}(t)$, we obtain that for all $t\in [0,\tau]$,
\vspace{-1mm}
$$ \frac{1}{2}\frac{\dd^+}{\dd t}\|z(t)-\tilde{z}(t)\|^2 \m\in\m 
   \left\langle A(z(t)-\tilde{z}(t)),z(t)-\tilde{z}(t)\right\rangle
   \vspace{-3mm}$$
$$ \qquad\qquad \ \  \quad \quad \quad \ \ \ \ \  \ \ -\left\langle 
   \Mscr z(t)-\Mscr \tilde{z}(t), z(t)-\tilde{z}(t)\right\rangle.
   \vspace{1mm}$$
The operator $\Mscr$ is maximal monotone by assumption (see Section
\ref{sec1}), and $A-\l I$ is dissipative (from
\rfb{Lax-Phillips_like}), therefore \vspace{-2mm}
$$ \half\frac{\dd^+}{\dd t}\|z(t)-\tilde{z}(t)\|^2 \m\leq\m \l\|z(t)
   -\tilde{z}(t)\|^2 \quad\quad \forall\ t\in [0,\tau].\vspace{-1mm}$$
Therefore $\|z(t)-\tilde{z}(t)\|^2\leq e^{2\l t}\|z(0)-\tilde{z}(0)
\|^2$ for $t\in[0,\tau]$. We know that $z(0)=\tilde{z}(0)$, therefore
$z(t)=\tilde{z}(t)$ for all $t\in [0,\tau]$. Hence \vspace{-5mm}
\begin{equation} \label{CausalSigma}
   \m\ \ \ \ \ \ \ \ \Sigma^{\rm st}_t\bbm{z_0\\ v} \m=\m 
   \Sigma^{\rm st}_t \bbm{z_0\\ \PPP_{\tau} v}\quad \quad \forall\ t
   \in [0,\tau], \vspace{-2mm}
\end{equation}
for all $\sbm{z_0\\v}\in \Dscr(\Ascr)$ with $v(\tau)=0$. Such pairs
$\sbm{z_0\\ v}$ are dense in $X\times \Uscr$, therefore by continuous
extension we get that \rfb{CausalSigma} remains true for all
$\sbm{z_0\\ v}\in X\times\Uscr$ (causality). \vspace{-1mm}
 
Now we prove the existence of classical solutions of \rfb{Boris} for
any pair $(x_0,u)$ as described after \rfb{Lax-Phillips_like}. Since
$\Ascr_0=\Ascr^\Mscr+\l I$, where $\Ascr^\Mscr$ is maximal dissipative
and densely defined and $\l\geq 0$, the differential inclusion
\vspace{-2mm}
\begin{equation} \label{psi_inclusion}
   \dot{\psi}_t \m\in\m \Ascr_0 \psi_t \m,\vspace{-2mm}
\end{equation}
with initial state $\psi_0\in\Dscr(\Ascr)$, has a unique classical
solution for all $t\geq 0$ (according to Theorem \ref{Kato}). We
decompose $\psi_t=\sbm{x(t)\\ u_t}$, then it follows from 
\rfb{Lax-Phillip_Semi} that $u_0(t)=u_t(0)$, therefore we have:
\vspace{-2mm}
\begin{equation} \label{Xi_Jinping}
   \dot{x}(t) \m\in\m (A-\Mscr) x(t) + Bu(t),\vspace{-2mm}
\end{equation}
which is the same as \rfb{Boris}. Hence for any $u\in\Hscr^1_\loc
((0,\infty);U)$ and $x_0\in X$ such that $Ax_0+Bu(0)\in X$, the
differential inclusion \rfb{Boris} has a unique classical
solution. This classical solution is Lipschitz continuous on any
finite interval and right differentiable for all $t\geq 0$, according
to Theorem \ref{Kato}.

Finally, we have to prove the estimate \rfb{Sol_inequ}. We go back to
the original norm on $U$. For two solutions $x$ and $z$ of \rfb{Boris}
and their respective inputs $u$ and $v$, we have that \vspace{-2mm}
$$ \n\n\n\frac{\dd}{\dd t}\|x(t)-z(t)\|^2 \m=\m 2\left\langle 
   x(t)-z(t), \dot{x}(t)-\dot{z}(t)\right\rangle \vspace{-4mm}$$
$$ \quad  \m\qquad \in\m 2\left\langle x(t)-z(t),A(x(t)-z(t)) + B
   (u(t)-v(t)) \right\rangle$$
$$ \n\n\n\n - 2\left\langle x(t)-z(t),\Mscr x(t)-\Mscr z(t)
   \right\rangle \m.\vspace{1mm}$$ 
Using \rfb{equi_cond1} and the fact that $-\Mscr$ is dissipative, we
obtain \vspace{-2mm}
$$ \frac{\dd}{\dd t}\|x(t)-z(t)\|^2 \m\leq\m 2\l\|x(t)-z(t)\|^2+2
   \beta\|u(t)-v(t)\|^2.$$
Denoting $\theta(t)=e^{-2\lambda t}\|x(t)-z(t)\|^2$, we have
\vspace{-1mm}
$$\dot{\theta}(t)\m\leq\m e^{-2\lambda t}2\beta\|u(t)-v(t)\|^2.$$
By integrating the above inequality, we obtain \rfb{Sol_inequ}. \
$\square$

\begin{corollary} \label{GeneralizedSol_4}
With the assumptions and the notation of Theorem \ref{Gen_sol}, for
any $x_0\in X$ and any $u\in\LEloc U$, there exists a unique function
$x:[0,\infty)\rarrow X$ such that $(x,u)$ is a generalized solution of
\rfb{Boris} and $x(0)=x_0$. Let $x_0,z_0\in X$, $u,v\in \LEloc U$, let
$(x,u)$ be a generalized solution of \rfb{Boris} with $x(0)=x_0$ and
let $(z,v)$ be a generalized solution of \rfb{Boris} with $z(0)=z_0$.
Then for all $t\geq 0$, \rfb{Sol_inequ} holds.
\end{corollary}

This corollary follows from Theorem \ref{Gen_sol} by the density of
$\Dscr(\Ascr^\Mscr)=\Dscr(\Ascr)$ in $X\times \Uscr$ and continuous
extension.

\smallskip
{\color{blue}
\begin{proposition} \label{equi_cond}
If $\l\geq 0$ is such that $\l I-A$ has a bounded inverse, then for
any $\beta>0$, condition \rfb{Lax-Phillips_like} in Theorem
\ref{Gen_sol} is equivalent to:
\begin{equation} \label{Equiv_cond}
   \|B^*\phi\|^2\leq 4\beta\left\langle(\l I-A^*)\phi,\phi\right
   \rangle \quad \forall \phi \in \Dscr(A^*),
\end{equation}
where $A^*$ and $B^*$ are adjoints of $A$ and $B$ respectively.
\end{proposition}}

{\em Proof.} \m If \rfb{equi_cond1} holds, then denoting $z_0=(A-\l 
I)x_0+Bu_0$, we get that $J\geq 0$, where the cost function $J$ is
defined for all $z_0\in X$ and $u_0\in U$ by \vspace{-1mm}
$$ J \m=\m \beta\|u_0\|^2-\left\langle z_0,(\l I-A)^{-1}Bu_0\right
   \rangle + \left\langle z_0,(\lambda I-A)^{-1}z_0\right\rangle.
   \vspace{-1mm}$$
Differentiate $J$ with respect to $u_0$ in order to minimize it:
\vspace{-2mm}
$$ \frac{\partial J}{\partial u_0}\m=\m 2\beta u_0-B^*(\l I-A^*)^{-1}
   z_0=0.\vspace{-2mm}$$
Thus the minimum with respect to $u_0$ is attained at $u_0=\frac{1}
{2\beta}B^*(\l I-A^*)^{-1}z_0$. Substituting this $u_0$ into $J$, we
obtain the minimum of $J$ (which we denote by $J_{min}$) as follows:
\vspace{-2mm}
$$ 0\leq J_{min} = \left\langle z_0,(\l I-A)^{-1}z_0\right\rangle
   -\frac{1}{4\beta}\|B^*(\l I-A^*)^{-1}z_0\|^2 .\vspace{-1mm}$$
Define $(\l I-A^*)^{-1}z_0 =\phi $ where $\phi\in\Dscr(A^*)$, then
the above inequality becomes \rfb{Equiv_cond}.

Conversely, if \rfb{Equiv_cond} is true for any $\phi\in \Dscr(A^*)$,
then for $z_0=(\lambda I-A^*)\phi $ we get\vspace{-3mm}
$$ 0 \m\leq\m \langle z_0,(\l I-A^*)^{-1}z_0\rangle-\frac{1}{4\beta}
   \|B^*(\l I-A)^{-1}z_0\|^2.\vspace{-3mm} $$
The right-hand side above is $J_{min}$, the minimum of $J$ with
respect to $u_0\in U$. Since $J_{min}\geq 0$, it follows that for all
$z_0\in X$ and all $u_0\in U$\vspace{-2mm}
$$ J \m=\m \beta\|u_0\|^2-\left\langle z_0,(\l I-A)^{-1}B u_0\right
   \rangle + \left\langle z_0,(\l I-A)^{-1}z_0\right\rangle\geq 0.
   \vspace{-2mm}$$
Denoting $x_0=(A-\l I)^{-1}\left[z_0-Bu_0\right]$ it is easy to derive
\rfb{equi_cond1}, which is equivalent to \rfb{Lax-Phillips_like}. \
$\square$       

\begin{remark} \label{bounded_B} \rm
For systems as in \rfb{Boris} with $B$ bounded (i.e. $B\in\Lscr(U,
X)$) and with $A$ dissipative, the inequality \rfb{Equiv_cond} will
always hold for some $\beta>0$, $\l>0$. This follows easily if we
take $4\beta\l\geq\|B\|$. This happens to be the case in Example
\ref{SCOLE_ex} for the nonlinear system described by \rfb{eq:SCOLE}.
\end{remark}

\begin{remark} \label{Admissible_B} \rm 
The estimate \rfb{Lax-Phillips_like}, or equivalently
\rfb{Equiv_cond}, implies that $B$ is an admissble control operator
for $\tline$, the operator semigroup generated by $A$. This follows
from \rfb{Sol_inequ} by taking $\Mscr=0$, $x_0=0,\m z_0=0,\m v=0$, or
alternatively, from Theorem 5.1.1 (part (d)) in \cite{obs_book}.
\end{remark}

\section{The main result} \label{sec6} 

First we have to clarify what we mean by a (possibly non-linear)
time-invariant well-posed system. Our definition is far from being the
most general or ``the best'' in any sense. Indeed, we assume that the
input, state and output spaces are real Hilbert spaces and the input
and output functions are of class $L^2_\loc$, because this is the
framework that we are used to from the linear time-invariant case, but
entirely different frameworks are conceivable. \vspace{-1mm}

Giving an axiomatic definition of a well-posed system in the spirit of
\cite{Weiss10} is possible but cumbersome. We prefer to define a
well-posed system via its (non-linear version of the) Lax-Phillips
semigroup. We now give an extension of Definition 5.6, to also include
outputs of the nonlinear system.\vspace{-1mm}

We regard $L^2([0,t];Y)$ as a subspace of $L^2([0,\infty);Y)$ (by
extending functions to be zero outside $[0,t]$). Recall the notation
$\PPP_\tau$ from Section \ref{sec2} and $\Uscr$ and $\Yscr$ from
Section \ref{sec4}. \vspace{-1mm}

\begin{definition} \label{abstractnonlinear} 
A time invariant well-posed (possibly nonlinear) system $\Sigma^\NL$
with input space $U$, state space $X$ and output space $Y$ consists of
two families of (possibly nonlinear) continuous operators\vspace{-2mm}
$$ \Sigma^{\rm  st} \m=\m (\Sigma^{\rm  st}_t)_{t\geq 0} \m,\quad
   \Sigma^{\rm out} \m=\m (\Sigma^{\rm out}_t)_{t\geq 0} \m,
   \vspace{-1mm}$$
where $\Sigma^{\rm st}_t:X\times\Uscr\rarrow X$ and $\Sigma
^{\rm out}_t:X\times\Uscr\rarrow L^2([0,t];Y)$ such that the 
following is a strongly continuous semigroup of (possibly nonlinear)
operators $\ULax^\NL=(\ULax^\NL_t)_{t\geq 0}$ acting on $\Yscr
\times X\times\Uscr$: for every $t\geq 0$, \vspace{-2mm}
\begin{equation} \label{J_Pasha}
   \ULax^\NL_t \m=\m \bbm{\Sscr_{-t} & 0 \sbluff    & 0 \\
   0 & I & 0 \\ 0 & 0 & \SSS_t^*} \left[\begin{array}{c|c c} I & \ 
   \Sigma^{\rm out}_t \\ 0 & \ \Sigma^{\rm st} _t \\ \hline
   \ \ 0 \ \  & 0\ \ \ I \end{array}\right].\vspace{-2mm}
\end{equation}
Moreover, we require that for all $\tau\geq 0$, \vspace{-3mm}
\begin{equation} \label{Identities}
   \Sigma^{\rm st}_\tau\bbm{x_0\\ v} =\m \Sigma^{\rm st}_\tau
   \bbm{x_0\\ \PPP_\tau v},\ \ \ \Sigma^{\rm out}_\tau\bbm{x_0\\ v}
   =\m \Sigma^{\rm out}_\tau\bbm{x_0\\ \PPP_\tau v},\vspace{-3mm}
\end{equation}
for all $v\in\Uscr$ and $x_0\in X$.
\end{definition} \vspace{-1mm}

The identities \rfb{Identities} are called the {\em causality
conditions}, the first one appeared previously in Definition
\ref{AbstractSystem}.\vspace{-1mm}

To understand the meaning of the above definition, one should compare
it to Proposition \ref{LaxPhilProp}. We have already encountered the
family $\Sigma^{\rm st}$ in Definition \ref{AbstractSystem}. It is
clear that in the linear case, we have the
decompositions: \vspace{-2mm}
\begin{equation} \label{Duke_of_Edinborough}
   \Sigma^{\rm st} _t \m=\m [ \tline_t\ \m\Phi_t]   \m,\qquad
   \Sigma^{\rm out}_t \m=\m [ \Psi_t  \ \m\fline_t] \m,\vspace{-2mm}
\end{equation}
but in the nonlinear case, in general we cannot split these operators
as above. The semigroup property for $\ULax^\NL$ implies functional
equations for the families of operators $\Sigma^{\rm st}$ and
$\Sigma^{\rm out}$ which, in the linear case, become the functional
equations in the definition of a well-posed linear system, as given
in \cite{Weiss10}. The resulting conditions for the family
$\Sigma^{\rm st}$ are precisely those in Definition
\ref{AbstractSystem}. Naturally, we call $\ULax^\NL$ the {\em
Lax-Phillips semigroup} of the system \m $\Sigma^\NL=(\Sigma^{\rm
st},\Sigma^{\rm out})$. \vspace{-1mm}

The causality condition must be explicitly verified because it does
not follow from the other assumptions in the definition. This can be
seen in the following trivial example.\vspace{-1mm}

\begin{example} {\rm
Consider the system \vspace{-2mm}
$$ \dot{x}(t) \m=\m Ax(t)+Bu(t)+B_1u(t+1) \vspace{-2mm}$$ 
$$ y(t) \m=\m Cx(t)+Du(t)+D_1u(t+7)$$
where $A,B,C,D,B_1,D_1\in \rline$. It is clear that the solutions of
the above equations are given by families of bounded linear operators
$\tline,\Phi,\Psi$ and $\fline$ such that \vspace{-3mm}
$$ \bbm{x(\tau)\\ \PPP_\tau y}\m=\m \bbm{\tline_\tau & \Phi_\tau\\ 
   \Psi_\tau & \fline_\tau}\bbm{x(0)\\ u}.\vspace{-3mm}$$
The reader can easily write the explicit formulas expressing these
operators. Defining $\Sigma^{\rm st}_t$ and $\Sigma^{\rm out}_t$ via
\rfb{Duke_of_Edinborough}, this system satisfies the properties in
Definition \ref{abstractnonlinear} except for the causality
\rfb{Identities}. This shows that causality does not follow from the
other properties required in Definition \ref{abstractnonlinear}.}
\end{example} \vspace{-1mm}

A time-invariant well-posed (possibly nonlinear) system $\Sigma^\NL$
is called {\em incrementally scattering passive} if its Lax-Phillips
semigroup $\ULax^\NL$ is contractive, or equivalently, if the estimate
\rfb{energy_balance_M} holds. Incrementally scattering passive systems
come with the big advantage that they can be described locally in
time, via the generator of $\ULax^\NL$. We shall discuss this in
another paper. \vspace{-1mm}

Here is our main result about systems of the form 
\rfb{Boris}-\rfb{Mirvis}:

{\color{blue}
\begin{theorem} \label{Hans}
Let \m $\Sigma$ be a scattering passive linear system with input space
$U$, state space $X$ and output space $Y$, described by the operators
$A,B,\bar{C}$ and $D$, as in \rfb{Ax+Bu}. Let $\Mscr$ be a
(set-valued) maximal monotone operator defined on all of $X$. Then
there exists a time-invariant well-posed nonlinear system
$\Sigma^\Mscr$ obtained by replacing $A$ in \rfb{Ax+Bu} by $A-\Mscr$,
so that \m $\Sigma^\Mscr$ is described by the differential inclusion
\vspace{-3mm}
\begin{equation} \label{ABMscr}
   \dot{x}(t) \m\in \m \bbm{A-\Mscr & B}\sbm{x(t)\\[1mm] u(t)} \m,
   \vspace{-3mm}
\end{equation}
and the same output equation as $\Sigma:$  \vspace{-3mm}
\begin{equation} \label{CDMscr}
   y(t)\m=\m \bbm{\bar{C} & D} \sbm{x(t)\\[1mm] u(t)}.\vspace{-3mm}
\end{equation}
Moreover, $\Sigma^\Mscr$ is incrementally scattering passive.
\end{theorem}}

{\em Proof.} \m Let $\ULax$ be the Lax-Phillips semigroup of \m
$\Sigma$ and let $\GothA$ be its generator, as described in
Proposition \ref{LaxPhillipsGenTh}. According to Proposition
\ref{attack_in_Halle} and the linear version of Theorem \ref{Putin},
$\GothA$ is maximal dissipative and densely defined. We introduce the
operator $\GothA^\Mscr$ as a perturbation of $\GothA$: \vspace{-2mm}
\begin{equation} \label{Gefen}
   \GothA^\Mscr \m=\m \bbm{\frac{\dd}{\dd\xi} & 0 & 0 \\
   0 & A-\Mscr & B\delta^*_0\\ 0 & 0 & \frac{\dd}{\dd\xi}} ,
   \vspace{-2mm}
\end{equation}
with the same domain $\Dom\GothA$, as described in Proposition
\ref{LaxPhillipsGenTh}. Recall that $\delta^*_0$ is the operator of
point evaluation at $0$. \vspace{-1mm}

In order to prove that $\GothA^\Mscr$ is maximal dissipative and 
densely defined, we split $\GothA^\Mscr$ as follows: \vspace{-2mm}
\begin{equation}\label{Split}
 \GothA^\Mscr \m=\m \underbrace{\bbm{\frac{\dd}{\dd\xi} & 0 & 0
   \\ 0 & A & B\delta^*_0 \\ 0 & 0 & \frac{\dd}{\dd\xi}}}_{\text{
   $\GothA$}} + \underbrace{\bbm{0 & 0 & 0 \\ 0 & -\Mscr & 0 \\
   0 & 0 & 0}}_{\text{$\tilde{\Mscr}$}}.\vspace{-2mm}
\end{equation}
We have seen that the first term on the right-hand side, $\GothA$ is
maximal dissipative and densely defined. The second term
$\tilde{\Mscr}$ is maximal dissipative and everywhere defined on
$\Yscr\times X \times\Uscr$, by assumption. Therefore,
$\Dscr(\GothA)\cap\left({\rm int}\Dscr(\tilde{\Mscr})\right)=\Dscr
(\GothA)$, which is dense. According to Theorem 1 of \cite{Rocka}, it
follows that $\GothA^\Mscr$ is maximally dissipative (and densely
defined) on $\Yscr\times X\times\Uscr$. \vspace{-1mm}

We introduce the single-valued dissipative operator $\GothA^0$ having
the same domain as $\GothA^\Mscr$, namely $\Dscr(\GothA)$, and for
each $[y_0\ x_0\ u_0]^\top\in\Dscr(\GothA)$, $\GothA^0[y_0\ x_0\
u_0]^\top$ is the vector with minimal norm in the closed and convex
set $\GothA^\Mscr[y_0\ x_0\ u_0]^\top$. According to Theorem
\ref{Crandall-Pazy} (Crandall-Pazy), $\GothA^0$ generates a
contraction semigroup $\ULax^\Mscr$ on $\Yscr\times
X\times\Uscr$. \vspace{-1mm}

Consider a trajectory of $\ULax^\Mscr$ that starts from a vector 
$[y_0\ x_0\ u_0]^\top\in\Dscr(\GothA)$: \vspace{-2mm}
$$ \left[ y_t\ x_t\ u_t \right]^\top \m=\m \ULax_t^\NL
   \left[ y_0\ x_0\ u_0 \right]^\top \FORALL t\geq 0.\vspace{-1mm}$$
According to Theorem \ref{Putin}, this trajectory remains in 
$\Dscr(\GothA)$ for all $t\geq 0$, which implies that for all $t\geq
0$ we have \vspace{-1mm}
\begin{equation} \label{Pierce_Morgan}
   y_t \in \Hscr^1((-\infty,0);Y),\quad u_t \in \Hscr^1((0,\infty);
   U),\vspace{-1mm}
\end{equation}
\begin{equation} \label{x_t,u_t(0)}
   \left[x_t \ u_t(0)\right]^T\in\Dscr(\AB) .
\end{equation} 
Moreover, still using Theorem \ref{Putin}, the trajectory $[y_t\ 
x_t\ u_t]^\top$ is locally Lipschitz continuous and right 
differentiable, and it satisfies for all $t\geq 0$ the differential
equations\vspace{-2mm}
\begin{equation} \label{systemevolution}
   \frac{\dd}{\dd t} \left[ y_t\ x_t\ u_t \right]^\top \m=\m 
   \GothA^0 \left[ y_t\ x_t\ u_t \right]^\top .\vspace{-2mm}
\end{equation}
The last row of \rfb{systemevolution} is the partial differential
equation $\frac{\partial}{\partial t}u_t=\frac{\partial}{\partial\xi}
u_t$ in $\Uscr$. It follows that $u_t=\SSS_t^* u_0$, i.e., the last
component evolves according to the left shift semigroup on
$\Uscr$. Hence, $u_t(0)=u_0(t)$.  Combining this with
\rfb{x_t,u_t(0)}, we get that $\sbm{x_t\\ u_0(t)}\in\Dscr(\AB)$ for
all $t\geq 0$.

We know from Proposition \ref{ScatteringPassive} that
\rfb{ScatteringCond} holds. Using \rfb{AB}, this clearly implies that
\vspace{-2mm}
$$ \left\langle\bbm{A & B \\ 0 & -\frac{1}{2} I} \bbm{x_0\\ v},
   \bbm{x_0\\ v}\right\rangle \leq\m 0 \n\n\n \FORALL \bbm{x_0\\ v}
   \in\Dscr(\AB).\vspace{-2mm}$$ 
Thus $A, B$ and $\Mscr$ satisfy the assumptions of Theorem
\ref{Gen_sol}, with $\l=0$ and $\beta=\half$. From the middle row of
\rfb{systemevolution}, we obtain that
$$\frac{\dd }{\dd t}x_t\in (A-\Mscr)x_t+Bu_t(0).$$
Since $u_t(0)=u_0(t)$, we get that $x_t$ satisfies \rfb{Boris} (with
$u_0$ in place of $u$). According to Theorem \ref{Gen_sol}, this
differential inclusion has a unique classical solution \vspace{-2mm}
$$ x_t \m=\m \Sigma^{\rm st}_t\sbm{x_0 \\[1mm] u_0}\ \FORALL t\geq 0,
   \vspace{-2mm}$$
where the operators $\Sigma^{\rm st}_t$ describe an abstract nonlinear
control system with input space $U$ and state space $X$. \vspace{-1mm}

The first row in \rfb{systemevolution} can be written in the form
$\frac{\partial}{\partial t}y_t=\frac{\partial}{\partial
\xi}y_t$. According to \rfb {Pierce_Morgan}, the boundary values
$w(t)=y_t(0)$ are well defined (in $Y$) for every $t\geq 0$. In terms
of these boundary values, $y_t$ can be expressed as
follows:\vspace{-2mm}
\begin{equation} \label{Navalny}
   y_t(\xi) \m=\m y_0(\xi+t)\ \ \mbox{ for }\ \ 
   \xi\in(-\infty,-t) \m,
\end{equation}
\begin{equation} \label{Protest}
   \n\n\n\n \n\n\n y_t(\xi) \m=\m w(\xi+t)\ \ \ \mbox{ for } \ \ 
   \xi\in[-t,0] \m,
\end{equation}
see for instance Example 10.1.9 in \citep{obs_book}. Since $y_t$ is a
continuous function of $\xi\in(-\infty,0]$, it follows that the
function $w$ must be continuous. From the description of
$\Dscr(\GothA)$ in Theorem \ref{LaxPhillipsGenTh} we see that
\vspace{-2mm}
\begin{equation} \label{w(t)}
   w(t) \m=\m \CD\sbm{x_t\\[1mm] u_0(t)},\vspace{-2mm}
\end{equation}
which is well defined according to \rfb{x_t,u_t(0)} and $u_t(0)=u_0
(t)$. We define the operator $\Sigma^{\rm out}_t$ that maps pairs 
$\left[x_0 \ u_0\right]^T\in X\times\Hscr^1((0,\infty);U)$ with the
property $\sbm{x_0\\u_0(0)}\in\Dscr(\AB)$, into $C([0,t],Y)$ by 
\rfb{w(t)}. Then \rfb{Navalny} and \rfb{Protest} imply that
\vspace{-1mm}
$$ y_t \m=\m \Sscr_{-t}\left(y_0+\Sigma^{\rm out}_t\sbm{x_0\\[1mm]
   u_0}\right).\vspace{-1mm}$$
Thus, we have verified that for $\left[y_0 \ x_0 \ u_0\right]^\top\in
\Dscr(\GothA)$, the semigroup $\ULax^\Mscr_t$ (determined by
$\GothA^\Mscr$) has the structure as required in \rfb{J_Pasha}. Since
$\Dscr(\GothA)$ is dense in $\Yscr \times X \times \Uscr$, the
operators $\Sigma^{\rm out}_t$ must have continuous extensions to $X
\times \Uscr$ and the same structure \rfb{J_Pasha} remains valid for
$\left[y_0 \ x_0 \ u_0\right]^\top\in \Yscr \times X \times \Uscr$. 
Since $\ULax^\Mscr$ is a semigroup of contractions, the system 
$\Sigma ^\Mscr$ is incrementally scattering passive. $\square$

\begin{remark} \label{single-valued M} \rm
Theorems \ref{Gen_sol} and \ref{Hans} can be modified by considering
a different class of nonlinear perturbations $\Mscr$. Suppose that
$\Mscr$ is a densely defined single-valued operator on $X$ that
satisfies the following: $Z\subset\Dscr(\Mscr)$ and there exist
constants $k_1\in [0,1)$, $k_2>0$ such that\vspace{-2mm}
\begin{equation} \label{our_estimate}
   \|\Mscr x_2-\Mscr x_1\|_X^2 \m\leq\m k_1^2\|x_2-x_1\|_Z^2 + k_2^2
   \|x_2-x_1\|_X^2 \m, \vspace{-2mm}
\end{equation}
for all $x_1,x_2\in Z$. The space $Z$ has been defined in
\rfb{SpaceZ}. Moreover, assume that $A-\Mscr$ is dissipative (for
instance, this is the case if $\Mscr$ is monotone). Then the
conclusions of Theorems \ref{Gen_sol} and \ref{Hans} are again
true. Indeed, in both proofs we have to replace Theorem 1 of
\citep{Rocka} with Theorem 4.2 of \citep{CranPazy}, in combination
with the lemma below, and all the other steps remains the same.
Theorem 4.2 of \citep{CranPazy} uses an estimate which, expressed in
our notation, is \rfb{Super_league} below. Thus, the proof of Remark
\ref{single-valued M} reduces to the following lemma. \vspace{-2mm}
\end{remark}
  
\begin{lemma}
We use the notation of Theorem {\rm\ref{Hans}} and the beginning of
its proof. Let $\Mscr$ be as in Remark {\rm\ref{single-valued M}}, in
particular, there exist $k_1\in[0,1)$ and $k_2\geq 0$ such that for
all $x_1,x_2\in Z$, the inequality \rfb{our_estimate} holds. Then
there exists $k_3\geq 0$ such that \vspace{-2mm}
\begin{equation} \label{Super_league}
   \n\nm\left\|\tilde{\Mscr}\sbm{y_2\\ x_2\\ u_2}-\tilde{\Mscr}
   \sbm{y_1\\ x_1\\ u_1}\right\| \m\leq\m k_1\left\|\GothA\sbm
   {y_2-y_1\\ x_2-x_1\\ u_2-u_1} \right\| + k_3\left\|\sbm{y_2-y_1
   \\ x_2-x_1\\ u_2-u_1} \right\|,\vspace{-1mm}
\end{equation}
for all $\left[ y_1\ x_1\ u_1\right]^\top$, $\left[ y_2\ x_2\ u_2
\right]^\top\in\Dscr(\GothA)$, where $\tilde{\Mscr}$ is defined as
in \rfb{Split}. 
\end{lemma}

{\em Proof}. \m For the sake of simplicity we assume that $0\in
\rho(A)$. By taking in \rfb{Chauvin_convicted} $\beta=0$, the norm 
on $Z$ becomes \vspace{-2mm}
$$ \|x\|_Z^2 \m=\m {\rm inf}\left\{\|A\varphi\|_X^2+\|v\|_U^2 \ 
   |\ x=\varphi - A^{-1}Bv \right\},\vspace{-2mm}$$
where $\varphi\in\Dscr(A)$, $v\in U$. For $x_1,x_2\in Z$, $x_1=
\varphi_1-A^{-1}Bv_1$, $x_2=\varphi_2-A^{-1}Bv_2$, $\varphi_1,
\varphi_2\in\Dscr(A)$, it easily follows 
that \vspace{-2mm}
$$ \|x_2-x_1\|_Z^2 \m\leq\m \|A(x_2-x_1)+B(v_2-v_1)\|_X^2 + 
   \|v_2-v_1\|_U^2 \m.$$
Therefore, from \rfb{our_estimate} we obtain that \vspace{-2mm}
$$ \|\Mscr x_2-\Mscr x_1\|^2 \m\leq\m k_1^2\|A(x_2-x_1)+B(v_2-v_1)
   \|_X^2 \vspace{-2mm}$$
$$ \quad\quad \quad\quad\quad \quad\quad\quad\quad+k_1^2\|v_2-v_1
   \|_U^2 + k_2^2\| x_2-x_1 \|_X^2 \m.$$
Let $u_1,u_2\in\Hscr^1((0,\infty);U)$ such that $u_1(0)=v_1$, 
$u_2(0)=v_2$. It is easy to see that  $\|v_2-v_1\|_U\leq\|u_2-u_1\|
_{\Hscr^1}$. Thus, \vspace{-2mm}
$$ \|\Mscr x_2 -\Mscr x_1\|^2 \m\leq\m k_1^2\|A(x_2-x_1)+B(u_2(0)-
   u_1(0)) \|_X^2 \vspace{-2mm}$$
\begin{equation} \label{Biden}
 \m\m\m\quad\quad \quad \ \ \ \quad\quad +k_1^2\|u_2-u_1
   \|_{\Hscr^1}^2+k_2^2\|x_2-x_1\|_X^2.
\end{equation}
For any $x_1,u_1$ as above there exist functions $y_1\in\Hscr^1
((-\infty,0);Y)$ such that $[y_1\ x_1\ u_1]^\top\in\Dscr(\GothA)$,
as can be understood from Theorem \ref{LaxPhillipsGenTh}. Similarly,
there exist $y_2\in\Hscr^1((-\infty,0);Y)$ such that $[y_2\ x_2\ 
u_2]^\top\in\Dscr(\GothA)$. Using such triples, we see that the 
left-hand side of \rfb{Biden} is the square of the left-hand
side of \rfb{Super_league}. It follows that \vspace{-2mm}
$$ \left\|\tilde{\Mscr}\sbm{y_2\\ x_2\\ u_2}\nm-\tilde{\Mscr}
   \sbm{y_1\\ x_1\\ u_1}\right\| \leq\m k_1\left\|\AB\sbm{ x_2-x_1\\
   u_2(0)-u_1(0)} \right\|_X \vspace{-3mm}$$
$$ \quad\quad \quad \ \ \ \quad\quad\quad\quad+k_1\|u_2-u_1
   \|_{\Hscr^1}+k_2\|x_2-x_1\|_X \m.$$
It is easy to see that the inequality \rfb{Super_league} follows.
$\square$\vspace{-1mm}

\begin{example}{\rm An example to illustrate Remark
\ref{single-valued M} is a delay line with the state space $X=L^2[0,1]$ and
with a {\it deadzone} perturbation. Consider the standard realization
of a delay line (see for instance \citep{Weiss10}), where $A=\frac{\dd
}{\dd \xi}$ with domain $\Dscr(A)=\left\{ x\in\Hscr^1(0,1)\m |\ x(1)=0
\right\}$, $B=\delta_1,\ \bar{C}=\delta^*_0$ and $D=0$. Define the 
deadzone function $d:\rline\rarrow\rline$ by $d(x)=0$ for $|x|\leq 1$,
$d(x)=x-1$ for $x>1$, $d(x)=x+1$ for $x<1$. The nonlinear operator
$\Mscr$ is defined by $(\Mscr x)(\xi)=m(\xi)d(x(\xi))$, where $m\in
L^2[0,1]$ such that $m\geq 0$ and $m\not\in L^{\infty}[0,1]$. We put
$\Dscr(\Mscr)=\Hscr^1(0,1)=Z$ (we could take a larger domain but it is
not needed). $\Mscr$ has no continuous extension to $X$ but it
satisfies the estimate \rfb{our_estimate} if $\|m\|_X$ is small
enough. \vspace{-2mm}

It is easy to see that $\Mscr$ is monotone. Then following Remark
\ref{single-valued M}, we can conclude that the nonlinear system
$\Sigma^\Mscr$ is well-posed and incrementally scattering passive.
}\end{example} \vspace{-1mm}

 \section{Conclusions} 

In this paper, we prove the well-posedness of a class of nonlinear
infinite dimensional systems, obtained as perturbations of scattering
passive linear systems. The problem is motivated from engineering
examples, for instance, the model of a wind turbine tower with a 
tuned mass damper in the nacelle, that we discuss in detail. Our main
result gives sufficient conditions for a nonlinear system described by
a differential inclusion \rfb{Boris} and an output equation
\rfb{Mirvis}, to be well-posed and incrementally scattering
passive. \vspace{-2mm}

\bibliographystyle{plain}        

\end{document}